\documentclass[a4paper,11pt]{amsart}
\usepackage[english]{babel}
\usepackage{wrapfig}
\usepackage{graphicx}
\usepackage{subfigure}
\usepackage{amsmath}
\usepackage{amssymb}
\usepackage{amsfonts}
\usepackage{mathrsfs}
\usepackage{amsthm}
\usepackage{fancyhdr}
\usepackage{enumerate}
\usepackage{braket}
\usepackage{youngtab}
\usepackage{tikz}
\usepackage{listings}
\usetikzlibrary{snakes}
\usetikzlibrary{shapes.geometric}
\usepackage[all,cmtip,2cell]{xy}
\usepackage{color}   
\definecolor{darkgreen}{rgb}{0.0, 0.5, 0.0}
\usepackage{hyperref}
\hypersetup{
    colorlinks=true, 
    linktoc=all,     
    linkcolor=blue,  
    citecolor=darkgreen,
}

\usepackage{geometry}
\usepackage{color}

\theoremstyle{plain}                    
\newtheorem{theo}{Theorem}[section]      
\newtheorem{conj}[theo]{Conjecture}    
\newtheorem{prop}[theo]{Proposition}    
\newtheorem{cor}[theo]{Corollary}       
\newtheorem{lem}[theo]{Lemma}           
\theoremstyle{definition}              
\newtheorem{defin}[theo]{Definition}    
\newtheorem{ex}[theo]{Example}         
\theoremstyle{remark}                  
\newtheorem{remark}[theo]{Remark}          

\newcommand{\z}{\mathbb{Z}}
\newcommand {\C}{\mathbb{C}}
\newcommand {\R}{\mathbb{R}}
\newcommand{\gm}{\mathbb{G}_m}
\newcommand {\sln}{\textnormal{SL}_2(N)}
\newcommand{\zz}{\mathbb{Z}/2\mathbb{Z}}
\newcommand{\px}{(p\in X)}
\newcommand{\qz}{(q\in Z)}
\newcommand{\yd}{(Y,D)}
\newcommand{\oyd}{(\bar{Y},\bar{D})}

\newcommand{\bsig}{(B,\Sigma)}
\newcommand{\xres}{\tilde{X}}

\newcommand{\pl}{\mathbb{P}^1}
\newcommand{\Pic}{\textnormal{Pic}}
\newcommand{\Hom}{\textnormal{Hom}}
\newcommand{\Nef}{\textnormal{Nef}}
\newcommand{\NE}{\textnormal{NE}}

\newcommand{\specP}{\textnormal{Spec}\ \C [P]}

\newcommand{\aut}{\textnormal{Aut}}

\title{$\zz$-Equivariant smoothings of cusps singularities}
\author{Angelica Simonetti}
\date{}

\begin{document}
\maketitle

\begin{abstract}
Let $\px$ be the germ of a cusp singularity and let $\iota$ be an antisymplectic involution, that is an involution such that there exists a nowhere vanishing holomorphic 2-form $\Omega$ on $X\setminus \{p\}$ for which $\iota^*(\Omega)=-\Omega$. Assume also that the involution is fixed point free on $X\setminus\{p\}$. We prove that a sufficient condition for such a singularity equipped with an antisymplectic involution to be equivariantly smoothable is the existence of a Looijenga (or anticanonical) pair $\yd$ that admits an involution free on $Y\setminus D$ and that reverses the orientation of $D$. This work also contains the proof of an analogue necessary and sufficient condition for the $\zz$-equivariant smoothability of simple elliptic singularities $p\in C(E)$ with $E$ an elliptic curve of degree $d\leq8$ and even equipped with a $\zz$-action.
\end{abstract}

\tableofcontents

\section*{Introduction}
Cusp singularities are a specific type of surface singularities: more precisely, a point $p$ on a complex algebraic surface $X$ is said to be a cusp singularity if the exceptional locus $E=\pi^{-1}(p)$ of its minimal resolution $\pi:\tilde X\rightarrow X$ is either an irreducible rational nodal curve (with a single node) or a cycle of smooth rational curves meeting transversally. Every cusp singularity $\px$ has an associated dual cusp: we will refer to the exceptional locus $D$ of this dual cusp singularity as the cycle dual to $\px$. If the germ of a cusp singularity $\px$ admits an involution $\iota$ such that there exists a nowhere vanishing holomorphic 2-form $\Omega$ on $X\setminus \{p\}$ for which $\iota^*(\Omega)=-\Omega$, then we say $\iota$ is \emph{antisymplectic}. Equivalently, $\iota$ is antisymplectic if the induced involution on the minimal resolution reverses the orientation of $E$. Now let $\iota$ be an antisymplectic involution defined on $\px$ which moreover is fixed point free away from $p$ and consider the associated quotient. This gives a new singularity which is rational and log canonical (\cite{NW03}). Cusp singularities and their quotients by the action of $\zz$ are among the surface singularities which appear at the boundary of the compactification of the moduli space of surfaces of general type due to Koll\'ar, Shepherd-Barron and Alexeev (cfr. \cite{KSB88} and \cite{A94}). Since only those singularities that admit a smoothing family occur at the boundary of this moduli space, it is useful to find nice conditions under which they happen to be smoothable. This question has been answered completely when it comes to cusp singularities. Indeed in 1981 Loijenga proposed the following conjecture, also proving the necessity of the condition.
\begin{theo}[\cite{L81},\cite{GHK15a}]\label{looij}
A cusp singularity $\px$ is smoothable if and only if the dual cycle $D$ sits as an anticanonical divisor on a smooth rational surface.
\end{theo}
The proof of the sufficiency of this conjecture came later in a broader paper on the mirror symmetry of log Calabi-Yau surfaces by Gross, Hacking and Keel \cite{GHK15a}. This result is interesting because it connects the deformation theory of this type of singularities to the existence of certain surfaces, the Looijenga pairs which can be checked algorithmically (see for example \cite{F15} and \cite{FM83}).

Inspired by these results, the main goal of this paper is to address the same problem for quotient cusp singularities, or, to be more precise, to investigate under which conditions a smoothable cusp singularity is equivariantly smoothable with respect to the action of $\zz$. Indeed, not all the smoothings of the quotient cusp singularity $(q\in Z)$ come from $\zz$-equivariant ones of the associated cusp singularity. However, the quotient cusps appearing at the boundary of the KSBA compactification described above are only those admitting a $\mathbb{Q}$-Gorenstein smoothing, that is one determined by an equivariant smoothing of the cusp. Thus we will focus our attention on this type of smoothing families.

We have been able to give a sufficient condition for cusp singularities to be $\zz$-equivariantly smoothable, which can be stated as follows:

\begin{theo}\label{suff}
Let $\px$ be the germ of a cusp singularity equipped with an antisymplectic involution $\iota$ which is fixed point free away from $p$ and let $D$ be its dual cycle. If there exists a Looijenga pair $\yd$ endowed with an antisymplectic involution that extends the one induced on $D$ by $\iota$ and is fixed point free away from $D$, then the cusp singularity $\px$ is equivariantly smoothable.
\end{theo}

Here an involution $j:\yd\rightarrow\yd$ is said to be antisymplectic it reverses the orientation of $D$. The proof of this result is based on the work of Gross, Hacking and Keel (\cite{GHK15a}): we use the involution defined on the surface $Y$ to get an equivariant version of the GHK construction. From this family we then obtained the required equivariant smoothing for the cusp singularity. Theorem \ref{suff} is already very useful. Indeed it allows to prove the following interesting fact.
\begin{cor}\label{nless10}
All cusp singularities of multiplicity $n\leq 10$ admitting an antisymplectic involution which is fixed point free away from $p$ are equivariantly smoothable.
\end{cor}

In order to prove these results, a great importance had the study of the other main character of theorem \ref{suff}, that is Looijenga pairs. A Looijenga pair is a smooth projective surface $Y$, together with an anticanonical divisor $D$ which is either an irreducible rational curve with a single node or a cycle of smooth rational curves. Examples of Looijenga pairs are provided for instance by the toric surfaces $\mathbb P^2$ and $\pl\times\pl$ with their respective toric boundaries (or, more generally, smooth toric surfaces with their toric boundaries). 

We studied Looijenga pairs equipped with an \emph{antisymplectic involution} fixed point free on $Y\setminus D$, proving the following result
\begin{theo}\label{equivblowup} 
Let $(\pl\times\pl, \Delta)$ be the toric Looijenga pair given by $\pl\times\pl$ together with its toric boundary $\Delta$. A negative definite Looijenga pair $\yd$ with $n\geq 4$ can be equipped with an antisymplectic involution $j$ that is fixed point free on $Y\setminus D$, if and only if there exists a sequence of contractions of disjoint pairs of $(-1)$ curves
\begin{equation}\label{contmaps}
\yd\overset{\psi_1}{\longrightarrow}(Y_1,D_1)\overset{\psi_2}{\longrightarrow}\dots\overset{\psi_{m-1}}{\longrightarrow}(Y_{m-1},D_{m-1})\overset{\psi_m}{\longrightarrow}(\pl\times\pl, \Delta)
\end{equation}
that respects the $\zz$-action defined on $\yd$ and induces on $(\pl\times\pl,\Delta)$ the action given by the map $j: (z,w)\mapsto (1/z,-w)$
\end{theo}
Moreover, if the length $n$ of $D$ is such that  $4\leq n \leq 10$ and $D$ is the dual cycle to a symmetric cusp singularity which admits an antisymplectic involution fixed point free on $X\setminus \{p\}$, then there always exists a smooth projective surface $Y$ containing $D$ as an anticanonical divisor and a $\zz$-action defined on it. This, together with theorem \ref{suff}, implies the result stated in corollary \ref{nless10}. 

Finally we would like to observe that theorem \ref{suff} is in fact part (the sufficient condition) of a more comprehensive conjecture, modeled on Looijenga's theorem for which we would like to find a complete proof in the coming years.
\begin{conj}\label{mainc}
Let $\px$ be a cusp singularity equipped with an antisymplectic involution $\iota$. Then $p\in X$ admits an equivariant smoothing if and only if the dual cycle $D$ sits as an anticanonical divisor on a smooth rational surface which admits an antisymplectic involution extending the one induced on $D$ by $\iota$ on $\px$.
\end{conj} 

Using theorem \ref{equivblowup}, one can see that among cusps of multiplicity equal to twelve there can be found examples of cusp singularities that are equipped with an antisymplectic involution and are smoothable, but for which there does not exist a Looijenga pair $\yd$ that admits an antisymplectic involution (fixed point free away from D) extending the action defined on $D$. Then conjecture \ref{mainc} would allow us to conclude that they are not equivariantly smoothable. The phenomenon described by the conjecture is well illustrated by what happens for simple elliptic singularities equipped with an action of $\zz$. It is well known that the deformation space of cones over elliptic curves $E$ of degree eight has two (essentially) different smoothing components: one of them has associated Milnor fibre  given by $M_{i}=\pl\times\pl\setminus E$ while the other one is associated to $M_{ii}=\mathbb F_1\setminus E$. Moreover their mirror surfaces are respectively $U_{i}=Y_{i}\setminus D$ and $U_{ii}=Y_{ii}\setminus D$, where $(Y_{i},D)$ and $(Y_{ii},D)$ are two distinct negative semidefinite Looijenga pairs and $D$ is a cycle of eight smooth rational curves of self intersection $-2$. In subsection \ref{ellsect} we have been able to prove that on the one hand while $(Y_{i},D)$ admits a $\zz$-action which is free away from $D$, it is not possible to construct such an action on $(Y_{ii},D)$. On the other hand, the Milnor fibre of any $\zz$-equivariant smoothing of a simple elliptic singularity of degree 8 is isomorphic to $M_i$.

This paper is structured as follows. The first section contains the results on cusp singularities admitting an antisymplectic involution, the second section deals with Looijenga pairs and contains the proof of theorem \ref{equivblowup}, among the others. The final section contains the main theorem, \ref{suff}, and the results on equivariant smoothability of cusps admitting an antisymplectic involution of multiplicity $n\leq 12$ and of simple elliptic singularities.

\subsection{Acknowledgments}
This paper is the result of part of my work as a Ph.D. student at the University of Massachusetts in Amherst and therefore it owes a great deal to the many interactions with the people - professors and collegues - I had the pleasure to meet there. I am very grateful to my advisor Paul Hacking for suggesting me the problem of equivariant smoothability of cusps and for the countless enligthing conversations, from which I have learnt a lot. I am also grateful to professors Jenia Tevelev and Eyal Markman and to Jennifer Li for the useful discussions and suggestions on the problem.

\section{Cusp singularities}
\subsection{Definitions and general results}\label{intro}
Let $\px$ be the germ of an isolated normal surface singularity. Let $\pi:\xres\longrightarrow X$ be its minimal resolution and $E=\pi^{-1}(p)$ the exceptional locus. We summarize some well known facts about this type of singularities, see for instance Looijenga \cite{L81} and Friedman \cite{FM83}.

\begin{defin}
We say that $\px$ is a \emph{cusp singularity} if the exceptional locus $E$ is a union of smooth rational curves meeting transversally, $E=\bigcup_{i=1}^{n} E_i$, with dual graph a cycle and $n\geq 2$ or a rational curve with one node.
\end{defin}
Note that the negative definiteness of the intersection matrix for $E$ and the fact that $\pi$ is a minimal resolution, imply the following three conditions:
\begin{itemize}
\item[i.] Each self intersection $-e_i={E_i}^{2}$ is such that $e_i\geq 2$
\item[ii.] There exists at least one $j$ such that $e_j\geq 3$
\item[iii.] If $E$ only has one irreducible component, then $-E^2\geq 1$
\end{itemize}
Moreover, the cycle of integers $(e_1,\dots,e_n)$ determines the analytic type of the cusp singularity; in other words cusp singularities are taut (\cite{La73}).\\
The \textit{multiplicity} of $\px$ is equal to 2 if $E^2=-1$, otherwise it is equal to $-E^2$; $n$ is called the \textit{length} of the cycle. Following Friedman, we will occasionally abuse notation and use $E$ to indicate the cycle of curves, the cycle of integers and the cusp singularity itself.\par\medskip

\begin{remark}\label{dualformula}Every cusp singularity comes with an associated \textit{dual cusp}: one way to describe it is in terms of its cycle of integers. If the cusp $\px$ is given by the cycle
$$(a_1,\underbrace{2,\dots,2}_{b_1},a_2,\underbrace{2,\dots,2}_{b_2},\dots,a_l,\underbrace{2,\dots,2}_{b_l})$$
then the dual cusp $D$ is obtained as:
$$(b_1+3,\underbrace{2,\dots,2}_{a_2-3},b_2+3,\underbrace{2,\dots,2}_{a_3-3},\dots,b_l+3,\underbrace{2,\dots,2}_{a_1-3})$$
unless the length of the cycle is 1 or $E^2=-1$. In these cases we have:
\begin{itemize}
\item If $E=(1)$, then $D=(1)$
\item If $E=(e)$ with $e\geq 2$ then $D=(3,\underbrace{2,\dots,2}_{e-1})$
\item If $E=(3,\underbrace{2,\dots,2}_{e})$ with $e\geq 1$ then $D=(e+1)$
\end{itemize}
\end{remark}

The duality of cusp singularities $D,E$ can be described from various points of view, some of which will appear later in this section. To give an idea of how $E$ and its dual $D$ are related to each other we include the following result.
\begin{prop}[Lemma 1.4, \cite{FM83}]
Let $E$ represent a cusp singularity and $D$ represent its dual, then
\begin{itemize}
\item[i.] the dual to $D$ is $E$
\item[ii.] the length of $D$ is equal to $-E^2$
\end{itemize}
\end{prop}

Recall that the embedding dimension of a cusp singularity (more generally this holds true for all minimally elliptic singularities) $E$ is equal to $\max(3,-E^2)$. The proposition above implies that the latter is equal to $\max(3,\textnormal{length}(D))$, therefore we see that since we can have cusp singularities with an exceptional cycle of arbitrary length, then we can have cusp singularities of arbitrary embedding dimension. For cusps of multiplicity $m\leq 5$ we actually know more about the geometry of these embeddings: if $m\leq 3$ then $\px$ embeds in $\C^3$ as a hypersurface, if $m=4$ it embeds as a complete intersection in $\C^4$, finally if $m=5$ it embeds in $\C^5$ as the zero locus of the $4\times 4$ pfaffians of a $5\times 5$ skew matrix.

Cusp singularities have an interesting quotient construction which is due to Hirzebruch \cite{Hi73}. Let us describe the idea of this construction as it appears in \cite{GHK15a}. Let $N\cong \z^2$ and let $A\in \sln$ be a hyperbolic transformation, that is $A$ has a real eigenvalue $\lambda>1$. $A$ determines a pair of dual cusps as follows.\\
Choose two linearly independent eigenvectors $w_1,w_2\in N_\R\cong N\otimes \R$ for $A$ with eigenvalues respectively $1/\lambda$ and $\lambda$ so that $w_1\wedge w_2>0$ (with the standard counterclockwise orientation of $\R^2$). Let $C,C'$ be the interiors of the strictly convex cones spanned by $\{w_1,w_2\}$ and $\{w_2,-w_1\}$. We observe that $C,C'$ are invariant for $A$.\\
Now let $U_C,U_{C'}$ be the corresponding tube domains
$$U_C=\{z\in N_\C \ | \ \Im{z}\in C\}/N\subset N_\C/N\cong (\C^*)^2$$
$A$ acts freely and properly discontinuously on $U_C,U_{C'}$. Write $Y_C,Y_{C'}$ for the holomorphic hulls of $U_C/\Gamma,U_C'/\Gamma$ where $\Gamma$ is the subgroup of $\sln$ generated by $A$. At the level of sets $Y_C$ and $Y_{C'}$ are obtained from  $U_C/\Gamma,U_C'/\Gamma$ by adding one point to each of them, respectively $p\in Y_C,p'\in Y_{C'}$.

\begin{prop}[cfr. Chapter III, Section 2 of \cite{L81}]
$(p\in Y_C)$ and $(p'\in Y_{C'})$ are germs of  two cusp singularities which are dual to each other. Moreover all cusp singularities arise in this way.
\end{prop}
\begin{proof}[Idea of the proof]
Let $E$ represent a cusp singularity with associated cycle $(e_1,\dots,e_n)$. Then the matrix
$$A:=\begin{pmatrix}0 & -1\\ 1 & e_1\end{pmatrix}\begin{pmatrix}0 & -1\\ 1 & e_2\end{pmatrix}\cdots\begin{pmatrix}0 & -1\\ 1 & e_n\end{pmatrix}$$ will produce, through the process we described, the cusp $E$ and its dual $D$.
\end{proof}

\subsection{On the action of $\zz$ on a cusp singularity}
In this work we are interested in studying germs of cusp singularities that admit a $\zz$-action and their equivariant smoothings. Let us begin this section with the following definition.
\begin{defin}
Let $\px$ be the germ of a cusp singularity. An involution $\iota$ is \emph{antisymplectic} if there exists a nowhere vanishing holomorphic 2-form $\Omega$ on $X\setminus \{p\}$ for which $\iota^*(\Omega)=-\Omega$.
\end{defin}

Equivalently, the induced involution on the minimal resolution $\pi:\xres\longrightarrow X$ reverses the orientation of the exceptional cycle $E$. We observe that not all cusp singularities admit an antisymplectic involution that is also fixed point free on $X\setminus \{p\}$, as shown in the following proposition. This result is stated for instance in \cite{NW03}, but we will provide a brief proof for it below.

\begin{prop}\label{sym}
Let $\px$ be a cusp singularity that admits an antisymplectic involution $\iota$ that is fixed point free on $X\setminus \{p\}$. Then $E$ has the following properties:
\begin{itemize}
\item[i.] If $E=\bigcup_{i=1}^{n}E_i$, then each irreducible component $E_i$ is sent to some other irreducible component $E_{\sigma(i)}$ where $\sigma$ is a reflection in the dihedral group $D_n$. 
\item[ii.] None of the nodes in $E$ is fixed by $\iota$. Instead $\iota$ fixes setwise two of the irreducible components of $E$. In particular $n$ is even: without loss of generality we can always label the fixed components $E_n$ and $E_{n/2}$.
\item[iii.] Let $e_i$ be equal to ${-E_i}^{2}$. Then $e_n$ and $e_{n/2}$ are even and $e_i=e_{\sigma(i)}$ for all $i=1,\dots,n$. 
\end{itemize}
\end{prop}

\begin{proof}
Suppose there exists an involution $\iota$ as stated above.
It is clear that each $E_i$ has to be sent to some other $E_j$, where $j$ might be equal to $i$. This implies that the corresponding action on the dual graph has to be given by an element $\sigma$ of order 2 in the dihedral group of order $2n$. Moreover we claim that $\iota$ cannot fix a node of the exceptional cycle $E$. Indeed suppose it does, then we can choose local coordinates on a neighbourhood $U$ of the fixed node $p$ so that locally $E=(xy=0)$ and the action is given by the matrix
$$\begin{pmatrix} 0 & 1 \\ 1 & 0 \end{pmatrix}$$
It follows that $\iota$ fixes a line, which gives a contradiction to the original assumption. As a consequence the number of irreducible components of $E$ has to be even and the induced action on the dual graph has to be given by a reflection fixing 2 vertices, because we want the orientation of the cycle to be reversed. This proves (i) and (ii).
Finally, it is  easy to check that if $E_i$ is sent to $E_{\sigma(i)}$ then $e_i$ has to be equal to $e_{\sigma(i)}$. Let us focus now on $E_{n/2},E_n$, which according to our labeling are the components of $E$ fixed by the action: each of these curves is a rational curve with an involution defined on it. Since the automorphism we are considering does not fix these curves pointwise, then it has to fix 2 distinct points on each one of them. Consider the quotient $\tilde\rho:\xres\rightarrow \hat Z$ given by the involution. Here $E_{n/2},E_n$ are mapped to the rational curves $\hat F_{n/2}, \hat F_n$ and each curve contains two $A_1$ singularities corresponding to the fixed points: the minimal resolution of these singularities, $\tilde \pi:\tilde Z\rightarrow \hat Z$, is the composition of four blowups at the four distinct singular points on $\hat F_{n/2},\hat F_n$. For $i=n/2,n$, let $F_i\subset \tilde Z$ be the strict transform of $\hat F_i$ and $G_1,\dots,G_4$ be the exceptional divisors. We have 
$$F_{n/2}=\tilde\pi^*\hat F_{n/2}-\frac{1}{2}G_1-\frac{1}{2}G_2 \quad \textnormal{and} \quad F_{n}=\tilde\pi^*\hat F_{n}-\frac{1}{2}G_3-\frac{1}{2}G_4$$
which gives
$$F_{n/2}^2=\hat F_{n/2}^2+\left(\frac{1}{2}G_1\right)^2+\left(\frac{1}{2}G_2\right)^2=\hat F_{n/2}^2-\frac{1}{2}-\frac{1}{2}=\hat F_{n/2}^2-1$$
and similarly, $F_n^2=\hat F_n^2-1$. Therefore $\hat F_{n/2}^2=1/2E_{n/2}^2$ and $\hat F_{n}^2=1/2E_{n}^2$ have to be integers, which implies that $e_{n/2},e_n$ have to be even. Thus (iii) is verified.
\end{proof}

The proposition above and its proof give necessary conditions for a cusp singularity to admit an involution with the required properties, which can be summarized in the following definition.

\begin{defin}\label{symmdef}
We say a cusp $E$ is \textit{symmetric} if there exists a labeling of $E$ and a reflection $\sigma$ in the dihedral group of order $2n$ fixing $n$ and $n/2$ such that $E_n^2,E_{n/2}^2$ are even and $E_i^2=E_{\sigma(i)}^2$ for every $i=1,\dots,n$. In particular, the length of a symmetric cusp $E$ has to be even.
\end{defin}

\begin{remark}\label{symmetricdual} 
If a cusp $E$ is symmetric then the dual cusp $D$ is symmetric as well. This follows immediately from the way the cycle of integers for the dual cusp is constructed starting from the one of $E$. Indeed, let $n,m$ be respectively the lengths of $E$ and $D$ and consider $E_n$, one of the two curves fixed by $\sigma$ according to our convention on labels. Suppose that $e_n=-E_n^2>2$. Then, since $e_n$ is even, it produces an odd number of $(-2)$ curves in the dual cycle: call $D_m$ the central curve among them. Viceversa, if $e_n=2$, then $E_n$ is the central curve in a chain of $2l+1$  $(-2)$-curves for some integer $l$ and therefore it corresponds in $D$ to a curve of self intersection $-2l+2$: again, label this curve $D_m$. The same reasoning applied to $E_{n/2}$ gives a curve labeled $D_{m/2}$. Finally the symmetry of the remaining self intersections carries on to $D$, thus giving a $\zz$ action on the dual cycle which fixes exactly $D_{m/2}$ and $D_m$.
\end{remark}

The conditions given in definition \ref{symmdef} are also sufficient to construct an antisymplectic involution that is fixed point free on $X\setminus \{p\}$ on a cusp singularity. In the following proposition we state this result and construct an explicit involution on a given symmetric cusp, using the quotient construction of the singularity presented in section 1.1.

\begin{prop}\label{cuspact}
Given a cusp singularity $\px$, an antisymplectic involution that is fixed point free on $X\setminus \{p\}$ exists if and only if the associated exceptional cycle $E$ is symmetric.
\end{prop}

\begin{proof}
For this proof we refer to the quotient construction of a cusp singularity described in section \ref{intro}: let $A$ be the matrix associated to $\px$, let $w_1,w_2$ be a pair of eigenvectors in $N_{\mathbb{R}}$ for $A$. Recall that $C$ is the open cone generated by $w_1,w_2$ and $v_i$ for $i\in\mathbb{Z}$ are the lattice points on the boundary of $\Xi$, where $\Xi$ is the convex hull of the lattice points contained in $C$.\\
The forward direction of the statement follows from (i) and (iii) of the above proposition. Now suppose $E$ is symmetric and $n/2,n$ are the indices fixed by the reflection $\sigma$. It follows from the definition of a symmetric cusp that $e_n$ and $e_{n/2}$ are even and $e_i=e_{n-i}$ for $i\neq n$, $i\neq n/2$. Moreover, recall that we have the relation $e_0v_0=v_{-1}+v_1$ (where $e_0=e_n=E^2_n$). We can thus define an action on $N$ fixing $v_0$ as follows: choose $\{v_0,v_1\}$ as a basis (we can do that since the toric chart is smooth) for the lattice and let $B$ be the linear map given by the matrix
$$B=\begin{pmatrix}1 & e_0 \\ 0 & -1\end{pmatrix}$$
Then $B$ satisfies:
$$B^2=I, \quad Bv_0=v_0 \quad\textnormal{and} \quad Bv_1=e_0v_0-v_1=v_{-1}$$
Similarly, the involution $B$ maps each $v_i$ to $v_{-i}$: by induction, suppose $Bv_{i-1}=v_{1-i}$ and $Bv_{i-2}=v_{2-i}$, then $Bv_i=B(e_{i-1 \textnormal{ mod }n}v_{i-1}-v_{i-2})=e_{i-1 \textnormal{ mod }n}v_{1-i}-v_{2-i}$ and $e_{i-1 \textnormal{ mod }n}=e_{1-i \textnormal{ mod }n}$, thus $Bv_i=v_{-i}$. It follows that the cone $C$ is invariant under $B$, considered as a map on $N_\R$ and we can always rescale $w_1,w_2$ so that $Bw_1=w_2$. Thus $B$ induces an involution on $U_C=N_\R+iC/N$. Now let us consider the transformation $A\in\sln$ and the quotient $U_C/\Gamma$, where $\Gamma$ is the infinite cyclic group generated by $A$. For all $v_i\in N$ we have that $B(Av_i)=Bv_{i+n}=v_{-(i+n)}$ and $v_{-(i+n)}=A^{-1}(v_{-i})=A^{-1}(Bv_i)$. This holds true in particular for $v_0,v_1$, therefore we get the relation 
\begin{equation}\label{dihedraleq}BA(v+iw)=A^{-1}B(v+iw) \ \textnormal{for all} \ v+iw \in N_\C \end{equation} and $A,B$ generate an infinite dihedral group $\mathcal D_{A,B}$ (isomorphic to the semidirect product $\z\rtimes \zz$).

It will be useful for our proof to study the eigenvectors for $B$ more in detail. Observe that $v_0$ is an eigenvector for $B$ which belongs to the lattice $N$ and it is primitive. The second eigenspace for $B$ relative to the eigenvalue $-1$ is given by the equation $2x+e_0y=0$, therefore it is generated by the eigenvector $u_0=-\frac{e_0}{2}v_0+v_1$ which again belongs to $N$ because $e_0$ is an even integer and it is a primitive vector for this lattice. In fact the pair of eigenvectors $v_0,u_0$ forms a basis for the lattice $N$ because $\{v_0,v_1\}$ is a basis of $N$. To see this, it suffices to consider the change of basis given by the matrix $P$ described below:
$$P=\begin{pmatrix} 1 & \frac{e_0}{2} \\ 0 & -1\end{pmatrix}
\qquad \Rightarrow \qquad 
P^{-1}BP=\begin{pmatrix} 1 & 0 \\ 0 & -1\end{pmatrix}$$
Moreover, $u_0$ is contained in $C'$, the cone dual to $C$. In order to show this let us first write down $w_1,w_2$ in terms of $v_0,v_1$. The fact that $Bw_1=w_2$ and viceversa, the condition that $w_1\wedge w_2>0$ and the convention we use to label the vectors $v_i$ give us the following:
$$w_1=(\alpha+e_0\beta) v_0 -\beta v_1 \qquad w_2=\alpha v_0 +\beta v_1 $$
where $\alpha<0<\beta$ and $\alpha/\beta$ is irrational. Given that $u_0=-\frac{e_0}{2}v_0+v_1$, we can then write $u_0$ in terms of $w_1,w_2$, obtaining that $u_0=-(2\beta)^{-1}w_1+(2\beta)^{-1}w_2$. Since $\beta$ is positive, then we can conclude that $u_0\in C'$. 

Now, if we defined our involution on the cusp singularity using only the linear involution $B$, we would obtain a map that is not fixed point free on $Y_C\setminus \{p\}$. Therefore we compose $B$ with the translation by an element of the torus $t\in N_\C/N$. More precisely, since we still want to construct an involution, we have to choose a two torsion element (thus contained in $\frac{1}{2}N/N$). In order for the final involution $\vartheta_{B,t}$ to be fixed point free on the quotient of $U_C$ by the action of $A$, we need $\vartheta_{B,t}(v+iw)\neq A^{-k}(v+iw)$ for all $k\in\z$. Equivalently, we need $B(v+iw)+t\neq A^{-k}(v+iw)$ for all $k\in\z$, since we want $\vartheta_{B,t}$ to be defined as the composition of $B$ with the translation by $t$. Note that this is the same as asking that  $A^kB(v+iw)+t\neq v+iw$, because $At=t$ mod $N$, or equivalently $A(2t)=2t$ mod $2N$. This holds true since if $t\in \frac{1}{2}N/N$, then $2t$ lies in $N$, therefore the fact that $A$ is congruent to the identity matrix mod 2, as proved in lemma \ref{amod2}, immediately implies that $A(2t)=2t$ mod $2N$. Since there are two equivalence classes of reflections in the dihedral group generated by $A$ and $B$, then we only need to very this inequality for two representatives of this classes, as we do with the following reasoning.

Consider the isomorphism between $N_\C/N$ and $(\C^*)^2$ given by the exponential map. Fix $v_0,u_0$ as a basis for $N$. Then $B$ is associated to the matrix
$$\begin{pmatrix} 1 & 0 \\ 0 & -1 \end{pmatrix}$$
Thus, under the isomorphism between $N_\C/N$ and $(\C^*)^2$, the involution $B$ corresponds to the map $(x,y)\mapsto (x,y^{-1})$. In order for $\vartheta_{B,t}$ to be fixed point free we need the translation by $t$ to correspond to the map $(x,y)\mapsto(-x,\pm y)$, so that the pair $(B,t)$ acts on $(\C^*)^2$ via the map
$$\vartheta_{B,t}:(x,y)\mapsto(-x,\pm y^{-1})$$
Indeed the variable $x$ corresponds to the character $v_0^*$ and, thinking in terms of the minimal resolution of $p\in Y_C$, it corresponds to the direction which is normal to the divisor $E_n$. 
With respect to the basis chosen above, we want $t=av_0+bu_0$ such that $(a,b)+i(0,0) \mapsto (e^{a\cdot 2\pi i},e^{b\cdot 2\pi i})=(-1,\pm 1)$ in $(\C^*)^2$, thus we need $a=\frac{1}{2}$. Equivalently, $t$ cannot be equal to $0,\frac{1}{2}u_0$, that is it cannot be a pure multiple of $u_0$. The same reasoning has to be made for the other fixed component of $E$, $E_{n/2}$, thus $t$ cannot be a multiple of $u_{n/2}:=\frac{e_{n/2}}{2}v_{n/2}+v_{n/2+1}$. Therefore $t\in \frac{1}{2}N/N\setminus \{0,\frac{1}{2}u_0,\frac{1}{2}u_{n/2}\}$. Since $\frac{1}{2}N/N$ has order four, there is always an element $t$ that works.

We can now define our involution given by the pair $(B,t)$ on $N_\C/N$ as the composition of multiplication by $B$ and translation by $t$:
\begin{equation}\label{thetadef}\vartheta_{B,t}:\ v+iw\mapsto B(v+iw)+t=(Bv+t)+iBw\end{equation}
Observe that if $v+iw \in U_C$, meaning that $w\in C$, then $Bv+iBw$ still belongs to $U_C$ because the cone $C$ is invariant under $B$ and $Bv+t+iw\in U_C$ as well, since the translation by $t$ does not affect the imaginary part $w$. 
The involution $\vartheta_{B,t}$ is constant on the equivalence classes relative to the action of $A$, thus giving a well defined map on $U_C/\Gamma$. Indeed, following (\ref{thetadef}), $\vartheta_{B,t}(v+iw)=B(v+iw)+t$, thus $\vartheta_{B,t}(A(v+iw))=BA(v+iw)+t=A^{-1}B(v+iw)+t$. Note that in the last equality we used the dihedral relation (\ref{dihedraleq}). Summing up and using the fact that $At=t$ mod $N$ we get that 
\begin{equation*}\begin{split}
\vartheta_{B,t}(A(v+iw))&=BA(v+iw)+t \\
&=A^{-1}B(v+iw)+t \\
&=A^{-1}B(v+iw)+A^{-1}t \\
&=A^{-1}(\vartheta_{B,t}(v+iw))\end{split}\end{equation*}
mod $N$. Therefore we can conclude that indeed $\vartheta_{B,t}$ is constant on the equivalence classes.

Given what we discussed above, the pair $(B,t)$ defines a fixed point free map on the algebraic torus and therefore on $U_C$ that descends to a well defined and still fixed point free analytic map on the quotient $U_C/\Gamma$ which extends to its partial compactification $Y_C$ as an analytic map fixing the cusp singularity $p\in Y_C$.
\end{proof}

\begin{lem}\label{amod2}
Let $N$ be a two dimensional lattice and let $A$ be a transformation in $\textnormal{SL}(N)$ associated to a symmetric cusp. Then $A$ is congruent to the identity matrix mod 2N.
\end{lem}
\begin{proof}
A matrix $A$ associated to a symmetric cusp can always be written as a product of $n$ matrices. More precisely, if we fix $v_0,v_1$ as a basis for $N$ then $A$ is given by:
\begin{equation}\label{Amult}A=\begin{pmatrix}0 & -1 \\ 1 & e_1\end{pmatrix}\begin{pmatrix}0 & -1 \\ 1 & e_{2}\end{pmatrix} \cdots \begin{pmatrix}0 & -1 \\ 1 & e_{n/2}\end{pmatrix} \begin{pmatrix}0 & -1 \\ 1 & e_{n/2-1}\end{pmatrix}\cdots\begin{pmatrix}0 & -1 \\ 1 & e_1\end{pmatrix} \begin{pmatrix}0 & -1 \\ 1 & e_n\end{pmatrix}\end{equation}
since $e_{i}=e_{n-i}$ for $i=1,\dots,n$. Besides $e_{n/2},e_n$ are even integers, therefore the two corresponding matrices are congruent to the matrix
$$J=\begin{pmatrix}0 & 1 \\ 1 & 0\end{pmatrix}$$
mod 2. Now observe that for a fixed (positive) integer $e$, we have that
$$\begin{pmatrix}0 & -1 \\ 1 & e\end{pmatrix}J\begin{pmatrix}0 & -1 \\ 1 & e\end{pmatrix}=J$$
Therefore, mod 2, we get
$$\begin{pmatrix}0 & -1 \\ 1 & e_{n/2-1}\end{pmatrix} \begin{pmatrix}0 & -1 \\ 1 & e_{n/2}\end{pmatrix} \begin{pmatrix}0 & -1 \\ 1 & e_{n/2-1}\end{pmatrix}\equiv \begin{pmatrix}0 & -1 \\ 1 & e_{n/2-1}\end{pmatrix} J \begin{pmatrix}0 & -1 \\ 1 & e_{n/2-1}\end{pmatrix}= J$$
and, using this recursively in (\ref{Amult}) for each $i$ until $i=1$ we obtain
$$A\equiv J\cdot \begin{pmatrix}0 & -1 \\ 1 & e_n\end{pmatrix}\equiv J\cdot J=I \qquad \textnormal{mod}\ 2$$
Therefore $A$ is congruent to the identity matrix mod 2, as required.
\end{proof}

An involution on a cusp singularity $\px$ is determined by the $\zz$-action on the exceptional cycle $E$, in the following sense. Let $\px$ be a symmetric cusp singularity, let $E$ be its exceptional cycle and $\sigma$ a reflection acting on the components of $E$, as in definition \ref{symmdef}. Suppose we are given an antisymplectic involution $\iota$ which acts as the reflection $\sigma$ on $E$ and is fixed point free away from $p$. Let $\rho:\px\rightarrow\qz$ be the quotient map associated with $\iota$. Then $\qz$ is the germ of a rational isolated singularity, usually referred to as the cusp quotient singularity \cite{NW03} and the map $\rho$ gives an \'etale covering of $Z\setminus\{q\}$. Let $\tilde\rho:\xres\rightarrow \hat Z$ be the quotient of $\xres$ by the action corresponding to the involution $\tilde\iota$ on the minimal resolution of $\px$. We get the commutative diagram:
 $$\xymatrix{ E\subset\xres  \ar[r]^{\pi} \ar[d]_{\tilde\rho} & \px \ar[d]^{\rho}\\ \hat F\subset\hat Z \ar[r]_{\hat\pi} & \qz}$$
where $\hat F=\bigcup_{i=1}^{n/2-1} \hat F_i \cup \hat F_{n/2} \cup \hat F_n$ is the image of $E$ under $\tilde\rho$, more precisely $\tilde\rho^{-1}(\hat F_i)=E_i\cup E_{n-i}$ for $i\neq n/2,n$ while $\tilde\rho^{-1}(\hat F_i)=E_i$ for $i=n/2,n$. The proof of proposition \ref{cuspact} shows that $\hat Z$ contains four isolated $A_1$ singularities which lie in pairs on $\hat F_{n/2},\hat F_n$. The minimal resolution $\pi_Z:\tilde Z \rightarrow Z$ of the singularity $\qz$ is obtained by composing the minimal resolution $\tilde \pi:\tilde Z\rightarrow \hat Z$ of the $A_1$ singularities on $\hat Z$ with the map $\hat \pi$.
In particular the exceptional locus $F=\pi_Z^{-1}(q)$ is the union of the rational curves $\bigcup_{i=1}^{n/2-1} F_i \cup F_{n/2} \cup F_n \cup G_{1} \cup \dots \cup G_{4}$ where $G_{1},\dots, G_{4}$ are the exceptional divisors of the resolution of the $A_1$ singularities, $F_i=\tilde\pi^{-1}(\hat F_i)$ for $i\neq n/2,n$ and finally $F_{n/2},F_n$ are the strict transforms of $\hat F_{n/2},\hat F_n$ under $\tilde \pi$. The dual graph for $F$ is described in figure \ref{quotcusp}. 

\begin{figure}[h]
\begin{center}
\begin{tikzpicture}
\draw (-4,0.5) -- (-3,0);
\draw (-4,-0.5) -- (-3,0);
\draw (-3,0) -- (-2,0);
\draw (-2,0) -- (-1,0);
\draw [dashed] (-1,0) -- (1,0);
\draw (1,0) -- (2,0);
\draw (2,0) -- (3,0);
\draw (3,0) -- (4,0.5);
\draw (3,0) -- (4,-0.5);
\node at (-4,0.5) {$\bullet$};
\node at (-4,-0.5) {$\bullet$};
\node at (-3,0) {$\bullet$};
\node at (-2,0) {$\bullet$};
\node at (-1,0) {$\bullet$};
\node at (4,0.5) {$\bullet$};
\node at (4,-0.5) {$\bullet$};
\node at (3,0) {$\bullet$};
\node at (2,0) {$\bullet$};
\node at (1,0) {$\bullet$};

\node [above] at (-4,0.5) {$-2$};
\node [above] at (4,0.5) {$-2$};
\node [below] at (-4,-0.5) {$-2$};
\node [below] at (4,-0.5) {$-2$};

\end{tikzpicture}
\end{center}
\caption{Dual graph of the exceptional cycle $F$}\label{quotcusp}
\end{figure}
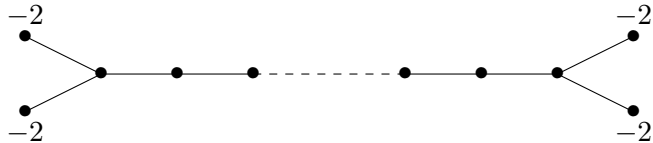

Note that the arrangement of the irreducible components of $F$ and their self intersections only depends on $\sigma$. Moreover, since quotient cusp singularities are taut (see for instance \cite{NW03}), then their isomorphism type is determined by the dual graph of the exceptional locus of the minimal resolution together with the self intersection numbers of its irreducible components. As a consequence, the quotient cusp $\qz$ only depends on the $\zz$-action defined by $\sigma$ on $E$. A priori though, the covering map $\rho:\px\rightarrow\qz$ could depend on the specific involution we consider for the germ of the singularity $\px$. The following proposition shows that in this case the map $\rho$ is the same up to isomorphism, for any choice of involution $\iota$ as long as the action of $\iota$ on $E$ is given by the same reflection, $\sigma$ and we only consider antisymplectic involutions which are free away from the cusp singularity.

\begin{prop}\label{uniqueinv}
Let $\iota$ be any antisymplectic involution defined on a cusp singularity $\px$  which acts as the reflection $\sigma$ on the associated exceptional cycle $E$ and is fixed point free away from $p$. Then the quotient map $\rho:\px\rightarrow\qz$ associated to $\iota$ coincides, up to isomorphism, with the index one covering map (as described for example in \cite{KM98}, Section 5.2, definition 5.19) of the quotient singularity $\qz$. That is, given two involutions with the properties described above, there exists an isomorphism $\vartheta:\px\rightarrow\px$ that makes the following diagram commute:

$$\xymatrix{ \px  \ar[r]^{\vartheta} \ar[d]_{\iota_1} & \px \ar[d]^{\iota_2}\\ \px \ar[r]_{\vartheta} & \px}$$
\end{prop}

\begin{proof}
Let us begin this proof with the observation that the index one covering map for $Z$ described in \cite{KM98}, definition 5.19 has domain the cusp $\px$ and is induced by a covering $\tilde X\rightarrow \hat Z$ which is \'etale on the smooth locus of $\hat Z$.
Now, the \'etale coverings of the smooth locus of $\hat Z$ are in bijective correspondence with the 2-torsion elements of its class group, $\textnormal{Cl}(\hat Z)$. Indeed, if $D\in \textnormal{Cl}(\hat Z)$ and $2D\sim 0$, then $\hat X:= \textnormal{Spec}_{\hat Z}(\mathcal O_{\hat Z}\oplus \mathcal O_{\hat Z}(D))$ with multiplication defined by an isomorphism $\theta: \mathcal O_{\hat Z}(D)^{\otimes 2}\rightarrow \mathcal O_{\hat Z}$ is a double cover of $\hat Z$ \'etale over the smooth locus and every such cover arises this way (cf. \cite{R76}, Cor 2.6). Note that the isomorphism type of the cover does not depend on the choice of the isomorphism $\theta$ because $\pi_1(\hat Z)=1$. Indeed, the map $\theta$ is determined up to a unit $u\in\textnormal{H}^0(\mathcal O_{\hat Z}^*)$: since $\pi_1(\hat Z)$ is trivial, $u$ admits a square root $v\in \textnormal{H}^0(\mathcal O_{\hat Z}^*)$. Then the map $\mathcal O_{\hat Z}\oplus \mathcal O_{\hat Z}(D)\rightarrow \mathcal O_{\hat Z}\oplus \mathcal O_{\hat Z}(D)$ defined by $(a,b)\mapsto (a,v\cdot b)$ induces an isomorphism of the double covers defined by $u\cdot\theta$ and $\theta$. Thus in order to study the map $\tilde \rho$ and therefore the quotient map $\rho$ it is useful to give a description of $\textnormal{Cl}(\hat Z)$. If  $G_{1},\dots,G_{4}$ are the exceptional divisors associated to the resolution of the four $A_1$ singularities of $\hat Z$, then the exact sequence 
$$0 \rightarrow \langle G_1,\dots,G_4\rangle\rightarrow \textnormal{Cl}(\tilde Z)\rightarrow \textnormal{Cl}(\hat Z)\rightarrow 0$$ and the fact that $\tilde Z$ is smooth give the isomorphism
$$\textnormal{Cl}(\hat Z)\cong \frac{\Pic (\tilde Z)}{\langle G_1,\dots,G_4\rangle}$$
Since $\qz$ is a rational singularity, then the Picard group of its minimal resolution is  the free abelian group $H^2(\tilde Z,\z)$ generated by the dual basis to the basis of $H_2(\tilde Z,\z)$ given by classes of the $k+4$ irreducible components of the exceptional locus (note that here $k=n/2+1$). Thus 
$$\textnormal{Cl}(\hat Z)\cong \z^{k+4}/Q\z^{4}$$
where $Q$ is the $(k+4\times 4)$ intersection matrix relative to $G_1,\dots,G_4$. This matrix $Q$ can always be put (permuting the rows) in the form
$$Q=\begin{pmatrix}-2& 0 & 0 & 0 & 1 & 0 & \cdots & 0 \\
						    0& -2 & 0 & 0 & 1 & 0 & \cdots & 0 \\
						    0& 0 & -2 & 0 & 0 & 0 & \cdots & 1 \\
						    0& 0 & 0 & -2 & 0 & 0 & \cdots & 1 \end{pmatrix}^T$$
Note that $Q$ gives the relations $-2G^*_1+F^*_{n/2}=0$, $-2G^*_2+F^*_{n/2}=0$, $-2G^*_3+F^*_{n}=0$, $-2G^*_4+F^*_{n}=0$, therefore $G^*_1-G^*_2$ and $G^*_3-G^*_4$ are elements of order two in $\textnormal{Cl}(\hat Z)$. The Smith normal form of $Q$ is the matrix
$$\begin{pmatrix}1& 0 & 0 & 0 & \cdots & 0 \\
					   0& 1 & 0 & 0 & \cdots & 0 \\
					   0& 0 & 2 & 0 & \cdots & 0 \\
					   0& 0 & 0 & 2 & \cdots & 0  \end{pmatrix}^T$$
which implies that $\textnormal{Cl}(\hat Z)\cong \z^k\times (\zz\times\zz)$, thus the class group of $\hat Z$ contains three non trivial 2-torsion elements corresponding to three possible covers of $\hat Z$ of degree two which are \'etale on the smooth locus, one of them giving the index one covering of $\qz$. 

Now, proposition \ref{cuspact} allows us to view the singularity $\qz$ as the partial compactification of the quotient of the tube domain $U_C$ by the action of the hyperbolic matrix $A$ and a pair $(B,t)$ described explicitly in the proof of the proposition. Thus the singularity $\qz$ is obtained by taking the quotient of $U_C$ by the action of the infinite dihedral group $\mathcal D_\infty$ generated by $a=A,b=(B,t)$.

\begin{equation} U_C \xrightarrow{/\langle a \rangle} \px \xrightarrow{/\langle b \rangle} \qz \label{q3}\end{equation}
The maps above can be understood completely by looking at the corresponding maps on neighborhoods $N,\tilde N$ of the exceptional locus $E$ and its universal covering $\tilde E$, which is a chain of rational curves indexed by $\z$. In particular the matrix $A$ acts on $\tilde E$ by translation giving the cycle of curves $E$ while the pair $(B,t)$ acts as a reflection, giving $\hat F$.

Each subgroup of index 2 of $\mathcal D_\infty$ corresponds to a covering of $\qz$ of degree 2, and therefore to a covering of $\hat Z$ which is \'etale on the smooth locus, and it is easy to see that there are three such subgroups: $H_1=\langle a^2,b\rangle$, $H_2=\langle a^2,ab\rangle$, $H_3=\langle a\rangle$. In particular, by the description in terms of the covering map above any such covering of $\hat Z$ arises in this way. In order to understand these coverings, as already stated above, it suffices to understand how the quotient maps induced by each subgroup $H_i$ act on $\tilde E \subset \tilde N$ and $E\subset N$. Clearly $H_3$ corresponds to the quotient maps given in (\ref{q3}). As for $H_1$, we get the following diagram
$$ \tilde E \xrightarrow{/\langle a^2 \rangle} E_1' \xrightarrow{/\langle b \rangle} E_1 \rightarrow \hat F$$
where $E_1'$ is a cycle of length $2n$, given that $n$ is the length of $E$, and $E_1$ is a chain of $n+1$ rational curves finally mapping to $\hat F$ through the action of $\overline a \in \mathcal D_\infty/H_1$. A very similar description holds true for $H_2$. Thus we see that the only subgroup giving us a covering of $Z$ by the cusp $X$, or alternatively of $\hat Z$ by $\tilde X$ is $H_3$. We deduce that this covering map is the one associated to the index one cover. 
\end{proof}

Given an antisymplectic involution on a cusp singularity that is fixed point free on $X\setminus \{p\}$ one may ask if it induces in a natural way an involution of the same type on the dual cusp. The answer is affirmative.

\begin{theo}\label{dualact}
Let $\px$ be a cusp singularity and let $\iota$ be an antisymplectic involution that is fixed point free on $X\setminus \{p\}$. Then $\iota$ induces an antisymplectic involution on the dual cusp $p'\in X'$ that is fixed point free on $X'\setminus \{p'\}$.
\end{theo}

\begin{proof}
Let $E$ be the exceptional divisor associated with $\px$, as usual. Observe that $E$ is symmetric, since it admits an involution. Then the theorem follows from remark \ref{symmetricdual} and proposition \ref{cuspact}.
\end{proof}

Finally, we can describe the relation between the involution on a cusp $\px$ and the one on its dual in light of lemma 7.3 from \cite{GHK15a}.

\begin{remark}\label{toricdualact}
As usual, let $(p\in Y_C)$ be a cusp singularity, $E$ its exceptional cycle. Let $\iota$ be the involution defined on it and $\sigma$ the reflection induced by $\iota$ on $E$. On the other side, let $(p'\in Y_{C'})$ be the dual cusp and $D$ its associated exceptional cycle. Finally let $G$ be the cycle of curves obtained from $E$ contracting all the $(-2)$ curves and let $F$ be the one obtained from $D$ through the same process. Note that this $F$ is exactly the divisor defined in lemma 7.3. Observe that $\sigma$ gives directly (by restriction to the indices that are left after the contraction) an involution $\sigma'$ on $G$ that respects the self intersections of the irreducible components of $G$. Now, because of the way $F,G$ are related to each other (in terms of their dual graphs, the vertices of $\Gamma_G$ are the edges of $\Gamma_F$ and viceversa), $\sigma'$ can be seen as a reflection on $F$ that respects self intersections, once we choose the appropriate labeling for it. Moreover $\sigma'$ extends from $F$ to $D$ thus giving a reflection $\sigma''$ on $D$. Using proposition \ref{cuspact}, $\sigma''$ induces an involution on $(p'\in Y_{C'})$ which is the one given by theorem \ref{dualact}.
\end{remark}

\section{Looijenga pairs}
An important role in the deformation theory of cusp singularities is played by Looijenga pairs. We present some preliminary facts on them following closely Friedman \cite{F15} and Gross, Hacking, Keel \cite{GHK15b}.
\begin{defin}
A \emph{Looijenga pair} $\yd$ is a smooth projective surface $Y$ together with a connected singular nodal divisor $D\in |-K_Y|$ which is either an irreducible rational curve with a single node or a cycle of smooth rational curves, $D=\sum_{i=1}^n D_i$, where each $D_i$ meets $D_{i+1}$ transversally, with $i$ understood mod $n$.
\end{defin}

We will also refer to pairs $\yd$ as \textit{anticanonical pairs}. The integer $n$ is called the length of $D$, if the components of $D$ are indexed as above, we refer to $\yd$ as a \textit{labeled Looijenga pair} and to the sequence of self intersections $(-D_1^2,-D_2^2,\dots,-D_n^2)$ as the \textit{cycle of integers} associated to it. To fix the notation, we will always label the components of $D$ starting from the top-right one, for instance, for $n=6$ we have:
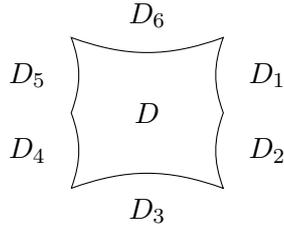
\begin{center}
\begin{tikzpicture}[out=20,in=160,relative]
\draw (-1,0) to (1,0) to (1,1) to (1,2) to (-1,2) to (-1,1) to  (-1,0);
\node [above] at (0,2) {$D_6$};
\node [right] at (1.2,1.5) {$D_1$};
\node [right] at (1.2,0.5) {$D_2$};
\node [below] at (0,0) {$D_3$};
\node [left] at (-1.2,0.5) {$D_4$};
\node [left] at (-1.2,1.5) {$D_5$};
\node at (0,1) {$D$};
\end{tikzpicture}
\end{center}
Note, as always, that all the pictures that will appear in this work are merely sketches: all components of $D$ should be understood as meeting transversally. An orientation of $D$ is an orientation of its dual graph, or equivalently the choice of a generator of $H_1(D,\z)\cong\z$. Observe that for $n\geq3$ an orientation determines a natural labeling of the components of $D$ up to cyclic permutation and viceversa a labeling induces an orientation on $D$.
\begin{lem}[Lemma 2.1, \cite{GHK15b}]\label{iso-orient}
Let $D$ be a cycle of $n$ rational curves, with a choice of orientation. This orientation induces an identification $\Pic^0(D)\cong \gm$, where $\Pic^0(D)$ is the group of numerically trivial line bundles.
\end{lem}
\begin{proof}[Construction of the isomorphism]
Let us describe how the isomorphism is obtained if $n\geq 3$. For $L\in\Pic^0(D)$ choose a nowhere-vanishing section $s_i\in \Gamma(L|_{D_i})$ for all $i$. Then define the map $\lambda$ as
$$\lambda:\Pic^0(D)\rightarrow \gm \qquad L\mapsto \prod_i \frac{s_{i+1}(p_{i,i+1})}{s_i(p_{i,i+1})}$$
where $p_{i,i+1}=D_i\cap D_{i+1}$. We have that $\lambda$ does not depend on the choice of the sections $s_i$ and it is an isomorphism.
\end{proof}
\begin{defin}\label{autyd}
An \textit{isomorphism} of labeled Looijenga pairs $(Y,D)$ and $(Y',D')$ is an isomorphism $f:Y\rightarrow Y'$ such that $f(D_i)=D'_i$ for each $i=1,\dots,n$ which is compatible with the orientation of $D$ and $D'$. Let $\aut \yd$ be the group of automorphisms of a labeled Looijenga pairs.
\end{defin}

If the intersection matrix $(D_i\cdot D_j)$ is negative definite, we call $\yd$ a \textit{negative definite Looijenga pair} and say that $D$ is negative definite. A useful invariant of anticanonical pairs is their charge.
\begin{defin}[\cite{F15}, Definition 1.1]
The \textit{charge} $Q\yd$ of a Looijenga pair is defined as
$$Q\yd=12-D^2-n$$
\end{defin}
To give a glimpse of the cohomology theory of anticanonical pairs, let $\Lambda\yd\subset H^2(Y,\z)$ be the orthogonal complement of the lattice spanned by the classes of the $D_i$. Then $\Lambda\yd$ is free (\cite{F15}, Lemma 1.5) and, if $D$ is negative definite (which implies that the classes $D_i$ are independent in cohomology), its rank is equal to the charge minus two (\cite{F15}, Lemma 1.5). We also note that in the case $D$ is negative definite, then $Q\yd\geq 3$ (\cite{F15}, Corollary 1.3).\\
Always with the aim of fixing our notation let us give the following definitions.
\begin{defin}
Let $\yd$ be a Looijenga pair. A curve $C$ in $Y$ is an \textit{interior curve} if none of its irreducible components is contained in $D$. An \textit{internal (-2)-curve} instead is a smooth rational curve of self intersection -2 that is disjoint from $D$. We say that $\yd$ is \textit{generic} if it has no internal (-2)-curves.
\end{defin}
Define a \textit{simple toric blowup} to be the blowup of a Looijenga pair $\yd$ at a node of $D$ and an \textit{interior blowup} to be a blowup of $Y$ at a smooth point on $D$. For a toric blowup $\widetilde{Y}\rightarrow Y$ define $\widetilde{D}=\sum_i \widetilde{D}_i+E$, where $\widetilde{D}_i$ is the strict transform of $D_i$ and $E$ is the exceptional divisor, while for an interior blowup set $\widetilde{D}=\sum_i \widetilde{D}_i$, where $\widetilde{D}_i$ is the strict transform of $D_i$. Then in both cases $(\widetilde{Y},\widetilde{D})$ is still a Looijenga pair. Interior blowups increase the charge $Q\yd$ by one, while corner blowups do not change it (\cite{FM83}, Lemmas 3.3 and 3.4). Finally we observe that the charge of a Looijenga pair $\yd$ has a topological interpretation: let $U=Y\setminus D$, then $e(U)=Q(Y,D)$ where $e(U)$ is the Euler number of $U$ (\cite{F15}, Lemma 1.2).

\begin{defin}
An \textit{antisymplectic involution} of a Looijenga pair $\yd$ is an involution $j:\yd\rightarrow \yd$ that reverses the orientation of $D$.
\end{defin}

\begin{remark}\label{inv}
We will be mainly interested in involutions of Looijenga pairs which are are antisymplectic and are fixed point free on $Y\setminus D$. Observe that if $j$ is such an involution then $j(D_i)=D_{\sigma(i)}$ where $\sigma$ is a reflection in the dihedral group of order $2n$.
\end{remark}

Now, recall that $\Lambda$ is defined to be the orthogonal complement of $D_1,\dots,D_n$ in $\Pic(Y)$. 
\begin{defin}
The canonical map
$$\phi_Y: \Lambda \longrightarrow\Pic^0(D)\cong\gm \qquad L\mapsto L|_D$$
determined by restriction of line bundles is called the \textit{period point} of $Y$
\end{defin}
Note that the period map $\phi_Y$ may depend on the labeling of $D$, since it is used in lemma \ref{iso-orient} for the construction of the isomorphism. Finally, note that there exists a version of Torelli's theorem for Looijenga pairs.
\begin{theo}[see \cite{GHK15b} or \cite{F15}]\label{torelli}
Let $\yd$, $(Y',D')$ be labeled Looijenga pairs with $D,D'$ of the same length and let $$\mu: \textnormal{H}^2(Y,\z)\rightarrow \textnormal{H}^2(Y',\z)$$
be an isomorphism of lattices. Then $\mu=f^*$ for an isomorphism $f$ of labeled Looijenga pairs compatible with the orientations if and only if the following hold:
\begin{itemize}
\item [a.] $\mu([D_i])=[D'_i]$ for all $i$.
\item [b.] $\mu(\Nef \ Y)=\Nef \ Y'$.
\item [c.] $\phi_{Y'}(\mu(q))=\phi_Y(q)$ for all $q\in\Lambda$.
\end{itemize}
Moreover if $f$ and $f'$ are two such isomoprhisms, then there exists a $\phi\in K(Y,D):=\textnormal{ker}(\textnormal{Aut}(Y,D)\rightarrow\textnormal{Aut}(\textnormal{Pic}\  Y))$ such that $f'=\phi\circ f$. Conversely, if $\phi\in K(Y,D)$, then $f'=\phi\circ f$ is an isomorphism from $Y$ to $Y'$ such that $(f')^*=\mu$.
\end{theo}

A direct consequence of this theorem is the next result that we will need later in chapter 3 for the proof of proposition \ref{ellpairs}.
\begin{theo}\label{torelliinv}
Let $\yd$ be a Looijenga pair and let 
$$\theta:  \textnormal{H}^2(Y,\z)\rightarrow  \textnormal{H}^2(Y,\z)$$
be an involution of lattices. Then $\theta=f^*$ for an automorphism $f$ of the Looijenga pair $\yd$ if and only if the following hold:
\begin{itemize}
\item [i.] $\theta([D_i])=[D_{\sigma(i)})]$ for all $i$.
\item [ii.] $\theta(\Nef \ Y)=\Nef \ Y$.
\item [iii.] $\phi_Y(\theta(q))=\phi_Y(q)^{-1}$ for all $q\in\Lambda$.
\end{itemize}
Moreover $f$ is an involution if $K(Y,D)$ is trivial.
\end{theo}

\begin{proof}
We want to reduce ourselves to theorem \ref{torelli}. Let $Y'$ be equal to $Y$ and define $D'$ by $D'_i:=D_{\sigma(i)}$. Then $(Y',D')$ is a labeled Looijenga pair. The map $\mu:\textnormal{H}^2(Y,\z)\rightarrow \textnormal{H}^2(Y',\z)$ given by $\mu(L):=\theta(L)$ gives an isomorphism of lattices such that $\mu([D_i])=\theta([D_i])=D_{\sigma(i)}=D'_i$, therefore condition (a) is equivalent to (i). Moreover the condition $\theta(\Nef \ Y)=\Nef \ Y$ directly implies that $\mu(\Nef \ Y)=\Nef \ Y'$, giving that (b) is equivalent to (ii). Finally consider $\phi_{Y'}$. Using lemma \ref{iso-orient}, we get that
$$\phi_{Y}:\Lambda \rightarrow \Pic^0(D)\rightarrow\gm, \qquad L\mapsto  \prod_i \frac{s_{i+1}(p_{i,i+1})}{s_i(p_{i,i+1})}$$
where the $s_i$'s are sections of $L|_{D_i}$ for each $i$. We can use the same sections to define $\phi_{Y'}$: indeed, define $s'_i$ to be $s_{\sigma(i)}$. Then
$$\phi_{Y'}:\Lambda' \rightarrow \Pic^0(D')\rightarrow\gm, \qquad L\mapsto  \prod_i \frac{s'_{i+1}(p_{i,i+1})}{s'_i(p_{i,i+1})}$$
Now, if $\sigma(i)=j$, for some $j$, then $\sigma(i+1)$ has to be equal to $j-1$, because we assumed that $\sigma$ is a reflection, thus it changes the orientation of $D$. This gives us
$$\frac{s'_{i+1}(p_{i,i+1})}{s'_i(p_{i,i+1})}=\left(\frac{s_j(p_{j-1,j})}{s_{j-1}(p_{j-1,j})}\right)^{-1}$$
therefore $\phi_{Y'}(q)=\phi_Y(q)^{-1}$ for all $q\in\Lambda=\Lambda'$. As a consequence, $\phi_{Y'}(\mu(q))=\phi_Y(\theta(q)^{-1}=\phi(q)$ and condition (c) is equivalent to (iii). Hence we can apply theorem \ref{torelli} and we get that $\theta=f^*$ for an isomorphim $f:(Y',D')\rightarrow \yd$ of Looijenga pairs if and only if conditions (i), (ii) and (iii) hold. By construction of $(Y',D')$ this isomorphism $f$ can in fact be viewed as an automorphism of the labeled Looijenga pair $\yd$ which reverses the orientation of $D$. Finally consider $f^2:\yd\rightarrow\yd$. It is an automorphism of $\yd$ mapping each $D_i$ to itself. Therefore, again by theorem \ref{torelli}, if $K(Y,D)$ is trivial then $f^2$ has to be equal to the identity map and $f$ is an involution of $\yd$.
\end{proof}

\begin{prop}\label{modulispaceinv}
Suppose $\theta$ is an isometry of $\Pic Y$ and $\sigma$ is a reflection in the dihedral group of order $2n$ such that $\theta([D_i])=[D_{\sigma(i)})]$ for all $i$ and $\theta(\textnormal{Nef}\ Y)=\textnormal{Nef} \ Y$. Let $S$ be the locus in the moduli space
$$T=\Hom(\Lambda,\C^*)$$
of (marked) pairs $(Y',D')$ deformation equivalent to $\yd$ for which there exists an isomorphim $j$ with $\theta=j^*$ and $j(D'_i)=D,_{\sigma(i)}$. Then we have
$$S\cap T^{\textnormal{gen}}=\{\phi\in T^{\textnormal{gen}} \ | \ \phi\circ\theta(q)=\phi(q)^{-1} \textnormal{ for all }q\in \Lambda\}$$
where
$$ T^{\textnormal{gen}}=T\setminus \bigcup_{\alpha \in \Phi} \{\chi\in T \ | \ \chi(\alpha)=1\}$$
and $\Phi$ is the set of roots in $\Pic Y$
\end{prop}

\begin{proof}
It follows from theorem \ref{torelliinv} and the structure of the moduli space of marked Looijenga pairs described in \cite{GHK15b}. 
\end{proof}

\subsection{Equivariant minimal model program for pairs $\yd$}
In this section, similarly to what we did for cusp singularities, we will characterize Looijenga pairs that admit an antisymplectic involution which is free on $Y\setminus D$. For this section we assume that the length $n$ of $D$ is greater than or equal to 4 and that $D$ does not contain any curves with self intersection $-1$. Consider a negative definite Looijenga pair and assume it admits an antisymplectic involution (as a Looijenga pair) $j$. Let $(Z,F)$ be the quotient induced by the action: we know that $Z$ contains four singularities of type $A_1$ lying in pairs on two of the irreducible components of $F$. Moreover note that we have $K_Z+F=0$ in $\textnormal{Cl}(Z)\otimes \mathbb Q$. Let us study minimal models for $(Z,F)$.
\begin{theo}\label{mmp}
Let $(Z,F)$ be as above. Then there exists a sequence of contractions $(Z,F)\rightarrow (Z_1,F_1)\rightarrow\cdots\rightarrow (Z_m,F_m)=(Z',F')$ such that  $(Z',F')$ satisfies:
\begin{itemize}
\item [i.] $(Z',F')$ is a minimally ruled surface with four singularities of type $A_1$.
\item [ii.] These singularities lie in pairs on two distinct fibres and on 2 distinct sections of self intersection equal to 0.
\item [iii.] $F'$ consists of 3 rational curves of self intersection 0, one of which is a fibre of the ruling, while the other two are sections and contain the four $A_1$ singularities.
\end{itemize}
\end{theo}
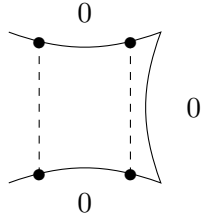
\begin{center}
\begin{tikzpicture}[out=20,in=160,relative]
\draw (-1,0) to (1,0) to (1,2) to (-1,2);
\draw[dashed] (-0.6,0.1) -- (-0.6,1.85);
\draw[dashed]  (0.6,0.1) -- (0.6,1.85);
\node [above] at (0,2) {0};
\node [right] at (1.2,1) {0};
\node [below] at (0,0) {0};
\node at (-0.6,0.1) {$\bullet$};
\node at (0.6,0.1) {$\bullet$};
\node at (-0.6,1.85) {$\bullet$};
\node at (0.6,1.85) {$\bullet$};
\end{tikzpicture}
\end{center}

Before proving this theorem let us state a more general result on the minimal model program for projective surfaces with $A_1$ singularities: this is already known in literature (see \cite{KM98}, Section 3.7, theorem 3.47), but we include it here for convenience of the reader.
\begin{theo}
Let $Z$ be a projective surface containing isolated singularities of type $A_1$. Then there exists a sequence of contractions $Z\longrightarrow Z'$ such that $Z'$ satisfies  one of the following
\begin{itemize}
\item[i.] $Z'$ has at worst $A_1$ singularities and $K_{Z'}$ is nef.
\item[ii.] $Z'$ has at worst $A_1$ singularities and it admits a map $Z'\overset{\varphi}{\longrightarrow} C$ where $C$ is a curve and the fibres of  $\varphi$ are smooth rational curves.
\item[iii.] $Z'$ is a Del Pezzo surface with at worst $A_1$ singularities and the Picard number is $\rho(Z')=1$.
\end{itemize}
\end{theo}

\begin{proof}
First, using the cone theorem in its generalized version (see for example \cite{KM98}, p.76, Theorem 3.7), we know that the contraction map $c_R:Z\rightarrow Z'$ exists for every extremal ray $R$ contained in the cone of curves of $Z$ such that $R\cdot K_Z<0$. Moreover if $C$ is a rational curve such that $[C]\in R$ then we get that $\rho(Z')=\rho(Z)-1$, moreover:
\begin{itemize}
\item [1.] If $C^2<0$, then $Z'$ has dimension 2. Here every curve whose class is contained in $R$ is contracted to one point $p$ and in fact the fiber over this point $p$, $c_R^{-1}(p)$, consist of one irreducible curve.
\item [2.] If $C^2=0$ then $Z'$ has dimension 1 and $\rho(Z)=2$. In this case an argument analogous to the one used for the smooth case shows that the fibres are connected and irreducible. Moreover they are still smooth and rational: let $F$ be a fibre. Then by assumption $F\cdot K_Z<0$ and, using the adjunction formula for the singular case we get $K_Z\cdot F + F^2=2p_a(F)-2+\textnormal{Diff}(Z,F)$, where $\textnormal{Diff}(Z,F)$, the different, is always a non negative quantity and $F^2=0$ (cfr. \cite{FA92}, section 16). Therefore we must have $2p_a(F)-2<0$, which implies that $p_a(F)=0$ so that $F$ is smooth and rational, as expected.

\item [3.] If $C^2>0$ then $Z'$ is a point and $\rho(Z)=1$.
\end{itemize}
Let us focus on case 1. There are only 2 types of curves $C$ satisfying $C^2<0$ and $C\cdot K_Z<0$: either (-1)-curves (as for smooth surfaces) or rational smooth curves passing through one surface singularity. To see this suppose that $C$ goes through two or more singularities. Resolve these singularities, and consider the corresponding map $\pi: \widetilde{Z}\rightarrow Z$. Let $E_1,E_2, \dots, E_l$ be the exceptional divisors and $\widetilde{C}$ the strict transform of $C$, then $$\widetilde{C}=\pi^*(C) -\mu_1 E_1-\dots -\mu_l E_l$$ where $\mu_1,\dots,\mu_l\in \frac{1}{2}\mathbb{Z}$ and they are non negative. We have, on one side
\begin{equation}\label{eq1}\widetilde{C}^2=C^2-2\mu_1^2-\dots-2\mu_l^2<0 \end{equation}
and on the other side
\begin{equation}\label{eq2}\widetilde{C}\cdot K_Z=C\cdot K_Z<0\end{equation}
Therefore, given that $\widetilde{Z}$ is smooth, \ref{eq1} and \ref{eq2} imply that $C$ is a (-1)-curve. Thus, by \ref{eq1}
$$-1=C^2-2\mu_1^2-\dots -2\mu_l^2 \implies -C^2=1-2\mu_1^2-\dots -2\mu_l^2$$
so that
$$0<1-2\mu_1^2-\dots -2\mu_l^2 \implies 2\mu_1^2+\dots+2\mu_l^2<1$$
The latter is impossible unless there is in fact only one exceptional curve $E$ and the corresponding $\mu=\frac{1}{2}$, in which case $\widetilde C\cdot E=\pi^*(C) -\frac{1}{2} E=-\frac{1}{2}(-2)=1$ hence $\widetilde C$ meets $E$ trasversally and there can be only one singularity on $C$. Locally $p\in C\subset Z$ is analytically isomorphic to $0\in(u=0)\subset \C^2/\frac{1}{2}(1,1)$.\par
The contraction of a (-1)-curve, exactly as for the smooth case, corresponds to a standard blow up. As for the second type of curve, the map $c_R:Z\rightarrow Z'$ is such that $c_R(C)=p$, where $p$ is a smooth point on $Z$. Indeed, suppose we resolve the singularity through $C$, we get a map $\pi:\widetilde{Z}\rightarrow Z$. As we did above, let $E$ be the exceptional divisor and $\widetilde{C}$ the strict transform of $C$: we can now first contract $E$ and then contract the image of $\widetilde{C}$ which has become a $(-1)$-curve, thus the composition of these two contractions, $\phi$, gives a new smooth surface and in particular $E,\widetilde{C}$ are mapped to a smooth point. 

$$\xymatrix{E\cup\widetilde{C}\ar@{|->}[d] \ar@/^2pc/[rrr]\ar@{}[r]|{\subset}& \widetilde{Z}\ar[r]^{\pi}\ar[d] & Z\ar[d] & C\ar@{|->}[d]\ar@{}[l]|{\supset}\\p \ar@{}[r]|{\in}& \overline{Z}\ar[r]^{\cong} & Z' & p'\ar@{}[l]|{\ni}}$$
Now the proof continues as in the usual (smooth) case: 
\begin{itemize}
\item We start with the surface $Z$. If $K_Z$ is a nef divisor then we stop: we have obtained the result stated in (i).
\item Otherwise there exists an extremal ray in the cone of curves of $Z$ whose intersection with the canonical divisor is negative: the contraction will produce one of the outcomes described at the beginning of this proof. If we are in case 2 or 3 then we stop and we get the result stated in (ii) or (iii).
\item If we are in case 1 then we go back to the first step and keep iterating the algorithm.
\end{itemize}
Note that given what we said about the possible types of curves that get contracted, we end up with a surface with at most the same number of singularities $Z$ had.
\end{proof}

\begin{remark}\label{cases23}
Let us analyze cases (ii) and (iii) more in detail. Firstly, there is only one singular Del Pezzo surface with Picard number equal to 1 only containing $A_1$ singularities: it is the weighted projective space $\mathbb{P}(1,1,2)$ (cfr \cite{D12}, chapter 8). \\
Now let us consider the second possible outcome of the minimal model program applied to $Z$. So far we know that we get a $\pl$-fibration $Z'\overset{\varphi}{\longrightarrow} C$ where $C$ is a curve, but we can actually say more about the arrangement of singularities along the fibers of this ruling: thanks to lemma 3.4 in \cite{KMcK99}, given a fibre $F$ then $Z'$ is smooth over $F$ or $F$ contains exactly two $A_1$ singularities or there's a unique singularity, which is a $D_n$ singularity, along $F$. Since the last case cannot happen with our initial assumptions, then if there are singularities along $F$ they must be exactly two and of type $A_1$.
\end{remark}

We can now go back to the theorem stated at the beginning of this section
\begin{proof}[Proof of theorem \ref{mmp}]
Given a pair $(Z,F)$ obtained as a quotient of a negative definite anticanonical pair $\yd$, run the minimal model program as described in the result we just proved: we know there are three possible outcomes for $Z'$. Since the surface we start with is rational with four singularities of type $A_1$, then it clearly has to be either as in (ii) or as in (iii). Moreover, there are no curves $C$ such that $C^2=-1/2$ and $K_Z\cdot C=-1$ at the same time. Indeed, suppose $C$ is such a curve: then we must have $(C\cdot F)_p=1/2(2k+1)$ where $p$ is the unique $A_1$ singularity on $C$ and $k$ is a non negative integer, since locally at the intersection we can write $C$ as the zero locus of some $f(u,v)=u^(2k+1)+\dots$ and the exponent of $l$ must be odd because $C$ is not Cartier at $p$. But this contradicts the hypothesis that $C\cdot F=C\cdot (-K_Z)=1\in\z$.

Therefore, running the minimal model program preserves all four singularities and, at the end, we get a new surface $(Z',F')$ still containing 4 $A_1$ singularities. Thanks to remark \ref{cases23} we can thus conclude that $(Z',F')$ is a surface admitting a ruling $\phi:Z'\rightarrow \pl$ such that each fiber is a smooth rational curve and the $A_1$ singularities lie in pairs on two distinct fibres and on two distinct components of $F'$. 

It remains to prove that the two components of $F$ containing the $A_1$ singularities are sections for the ruling and that there is only one more component in $F$ and it is a fibre. We know that $F'\cdot G=(-K_{Z'})\cdot G=2$, where $G$ is a fiber of the ruling. This implies that either $F'$ contains two sections or one bisection. In the first case we get that the two components containing the $A_1$ singularities $F'_1,F'_2$ are sections for the ruling. Moreover, in this case the irreducible components of $F'$ are exactly three and the third component is a fibre. Indeed, suppose that the two sections $F'_1,F'_2$ in $F'$ meet at point $q$ and call $G$ the fibre through that point. In this case we can always blowup at $q$ thus obtaining a new pair $(\hat Z,\hat F)$ where $\hat F$ is made of the strict transforms $\hat F_1,\hat F_2$ of $F'_1,F'_2$ plus the exceptional divisor $E$. Note that now the strict transform of $G$ has self intersection $(-1)$ and meets $E$ transversally but it is disjoint from $\hat F_1,\hat F_2$. Therefore we can contract the strict transform of $G$ and we get a pair $(Z'',F'')$ where $F''$ consists of the images of $\hat F_1,\hat F_2$ and the image of $E$ which has now self intersection 0 and is a fiber for the ruling on $(Z'',F'')$. Finally, we may always assume that the sections have self intersections equal to 0 by always contracting on the most negative section (cfr. the proof of theorem 2.1 in \cite{M90}). In the second case we need to be more careful. We have two possibilities.
\begin{itemize}
\item [a.] $F'$ consists of three irreducible components: one bisection $B$ and two fibres $G_1,G_2$ and the $A_1$ singularities lie on the fibres.
\item [b.] $F'$ consists of two components: the bisection $B$ and one fibre $G$. One pair of $A_1$ singularities lies on $B$ and one on $G$.
\end{itemize}
Let us assume we are in case (a). Consider the cover of degree two $f: (S\rightarrow \pl)\longrightarrow (Z'\rightarrow \pl)$, brunched at the four $A_1$ singularities, constructed via base change and normalization from a cover of degree two of $\pl$ by itself (branched the corresponding two points). Observe that by construction $S$ has to be a Hirzebruch surface $\mathbb{F}_n$ for some $n$ and that $f^{-1}(B)=B_1\sqcup B_2$, with each $B_i$ mapping isomorphically to $B$. Since $B_1\cdot B_2=0$, then $B^2=0$, therefore $S\cong\pl\times\pl$ and we can choose the other natural ruling of $\pl\times\pl$ to get the configuration we are interested in for the quotient $Z'$: the singularities on the (now) two sections in $F$ and a fibre as the third component. Now let us consider case (b). We start with the same $2:1$ cover described above. Let $\Gamma$ be the preimage of $B$ under this map and $\hat G$ the preimage of the fibre. Even in this case $S$ needs to be a Hirzebruch surface $\mathbb{F}_n$ for some $n$ and we have $-K_S\sim (n+2)\hat G+2B$, where $B$ is the negative section in $S$. On the other hand, $-K_S\sim \Gamma+\hat G$ and since $B$ can't be contained in $\Gamma$, we need to have $\Gamma\cdot B=(-K-\hat G)\cdot B=1-n>0$. This implies that $n$ is either 0 or 1. If $n=0$ we are back in the situation described in case (a) and we get the required conclusion. If $n=1$ instead, we get a contradiction. Indeed, suppose $n=1$: then $B$ is the only negative section in $S$, thus the $\zz$-action associated to the degree two covering has to send $B$ to itself. As a consequence $B$ has two fixed points (or it is fixed pointwise). But the action only has four fixed points, one for each singularity in $Z'$. Thus we get a contradiction. This concludes the proof.
\end{proof}

\begin{remark}\label{baseact}
Consider $\pl\times\pl$ along with its toric boundary 
$$\Delta=(\pl\times\{0,\infty\})\cup(\{0,\infty \}\times\pl)$$
Let $(z,w)$ be complex coordinates on $\pl\times\pl$ and set $\Delta_1=\{\infty\}\times \pl_w$, $\Delta_2=\pl_z\times\{\infty\}$, $\Delta_3=\{0\}\times \pl_w$, $\Delta_4=\pl_z\times\{0\}$. Define the map $j_0: (z,w)\mapsto (1/z,-w)$. Then $j_0$ is an involution with 4 fixed points $(1,0),(-1,0),(1,\infty),(-1,\infty)$ that interchanges $\Delta_1$ and $\Delta_3$ and preserves $\Delta_2,\Delta_4$. Thus $j_0$ is an involution of the labeled anticanonical pair $(\pl\times\pl,\Delta)$ and it defines a $\zz$-action on it which is free on $U=\pl\times\pl\setminus \Delta$.
\end{remark}

\begin{lem}\label{pllem}
The surface $(Z',F')$ in theorem \ref{mmp} is obtained from $\pl\times\pl$ as a quotient by the action defined in remark \ref{baseact}.
\end{lem}

\begin{proof}
Let $f$ be the involution on $\pl\times\pl$ such that the associated quotient space is given by $(Z',F')$. Then $f$ is induced by an automorphism of the algebraic torus $(\C^*)^2$: since $\textnormal{Aut}\ (\C^*)^2\cong \textnormal{GL}(2,\z)\rtimes (\C^*)^2$, then the involution on the torus has to be of the form $(B,t)$, where $B$ is a linear involution. More precisely
$$\begin{pmatrix}-1 & 0 \\ 0 & 1\end{pmatrix}$$
because of how it has to act on the toric boundary of $\pl\times\pl$. Therefore the map has to have the form $(x_1,x_2)\mapsto (t_1x_1^{-1},t_2x_2)$. Finally for this map to be an involution we must have $(t_1t_1^{-1}x_1,t_2^2x_2)=(x_1,x_2)$. It follows that on the one hand, we may always assume $t_1=1$ and on the other hand $t_2^2=1$ gives $t_2=\pm 1$. Since we want the involution to have only isolated fixed points, $t_2=-1$ and the claim is proved.
\end{proof}

We are now ready to state the main result of this section that gives a characterization of Looijenga pairs $\yd$, with $D$ of length at least four, admitting an antisymplectic involution fixed point free away from $D$.

\begin{theo}\label{desact} 
A negative definite Looijenga pair $\yd$ with $n\geq 4$ is equipped with an antisymplectic involution $j$ that is fixed point free on $Y\setminus D$, if and only if there exists a sequence of contractions of pairs of disjoint $(-1)$ curves
\begin{equation}\label{contmaps}
\yd\overset{\psi_1}{\longrightarrow}(Y_1,D_1)\overset{\psi_2}{\longrightarrow}\dots\overset{\psi_{m-1}}{\longrightarrow}(Y_{m-1},D_{m-1})\overset{\psi_m}{\longrightarrow}(\pl\times\pl, \Delta)
\end{equation}
that respects the $\zz$-action defined on $\yd$ and induces on $(\pl\times\pl,\Delta)$ the action defined in remark \ref{baseact}.
\end{theo}

\begin{proof}
Let us start assuming that $\yd$ is equipped with an antisymplectic involution. 
Using theorem \ref{mmp}  and lemma \ref{pllem} we get the diagram:

$$\xymatrix{(Z,F)\ar[r]^{\phi_1} & (Z_1,F_1)\ar[r]^{\phi_2} & \dots\ar[r] & (Z_{m-1},F_{m-1})\ar[r]^{\phi_m} &(Z',F')  \\ \yd\ar[u]^{p}&&&& (\pl\times\pl, \Delta)\ar[u]^{q}} $$
The map $\phi_1$ is the blow up of a single point $p$ in $(Z_1,F_1)$ lying on one of the irreducible components of $F_1$: let $E$ be the exceptional curve $E=\phi_1^{-1}(p)$ and consider the preimage of this (-1)-curve via the quotient map $p$, $\{E_1,E_2\}=p^{-1}(E)$. These two rational curves are in fact (-1)-curves which do not intersect: if they did, they would share a fixed point, hence $E$ would contain a singularity and this leads to a contradiction. Therefore we can subsequently contact $E_1$ and $E_2$. Let $(Y_1,D_1)$ be the composition of the contractions of $E_1,E_2$ in $\yd$. The diagram becomes

$$\xymatrix{(Z,F)\ar[r]^{\phi_1} & (Z_1,F_1)\ar[r]^{\phi_2} & \dots\ar[r] & (Z_{m-1},F_{m-1})\ar[r]^{\phi_m} &(Z',F')  \\ \yd\ar[r]^{\psi_1}\ar[u]^{p} & (Y_1,D_1)\ar[u]^{q_1} &&& (\pl\times\pl, \Delta)\ar[u]^{q}} $$
where the square on the left commutes. Repeating the same process $m-1$ times we obtain the sequence of maps:

$$\xymatrix{(Z,F)\ar[r]^{\phi_1} & (Z_1,F_1)\ar[r]^{\phi_2} & \dots\ar[r] & (Z_{m-1},F_{m-1})\ar[r]^{\phi_m} &(Z',F')  \\ \yd\ar[r]^{\psi_1}\ar[u]^{p}&(Y_1,D_1)\ar[r]^{\psi_2}\ar[u]^{q_1}&\dots\ar[r]&(Y_{m-1},D_{m-1})\ar[u]^{q_{m-1}}\ar[r] & (\pl\times\pl, \Delta)\ar[u]^{q}} $$
thus concluding the proof of necessity.
Viceversa, if there exists a sequence of contraction maps as in (\ref{contmaps}), then we can lift the involution $j_0$ defined in remark \ref{baseact} all the way up to the pair $\yd$, thus getting an antisymplectic involution fixed point free away from $D$ (see remark \ref{resact}).
\end{proof}

\begin{remark}\label{resact}
Let $\yd$ be a Looijenga pair and $j:\yd\rightarrow\yd$ an antisymplectic involution such that $j(D_i)=D_{\sigma(i)}$ with $\sigma$ a reflection in the dihedral group. Let the pair $\{p,j(p)\}$ be made of a point sitting on one of the irreducible components of $D$ and its image through $j$. Let $\pi:(\widetilde{Y},\widetilde{D})\rightarrow \yd$ be the composite map obtained by blowing up at $p$ and $j(p)$. Then the map $j$ lifts to a unique map $\widetilde{j}:(\widetilde{Y},\widetilde{D})\rightarrow (\widetilde{Y},\widetilde{D})$ (cfr. \cite{H77}, chapter 7) which is still an involution.
We explicitly note that if $p$ is a smooth point on $D$, so is $j(p)$, thus in this case $\widetilde{D}=\sum_i \widetilde{D}_i$ where $\widetilde{D}_i$ is the strict transform of $D_i$ and $\widetilde{j}(\widetilde{D}_i)=\widetilde{D}_{\sigma(i)}$.
Similarly if $p$ is a node, then $j(p)$ is a node as well, therefore $\widetilde{D}=\sum_i \widetilde{D}_i+E_p+E_{j(p)}$ where $E_p$ and $E_{j(p)}$ are the exceptional divisors of $\pi$, $\widetilde{j}(\widetilde{D}_i)=\widetilde{D}_{\sigma(i)}$ and $\widetilde{j}(E_p)=E_{j(p)}$. As a consequence $\widetilde{j}$ is still an involution of Looijenga pairs.
\end{remark}

\begin{remark}\label{equivtoricmod}
Theorem \ref{desact} implies that if a Looijenga pair $\yd$ admits an antisymplectic involution $j$, then there exists a toric model $\yd\rightarrow \oyd$ such that $\oyd$ admits an antisymplectic involution $\bar j$.
\end{remark}

Theorem \ref{desact} gives us a useful criterion to decide, given a cusp $D$, whether or not it is possible to find a smooth rational surface $Y$ where $D$ sits as an anticanonical divisor (in other words if there exists a negative definite Looijenga pair $\yd$ equipped with an involution $j$ which is free on $Y\setminus D$). Indeed, suppose you can. If there exists $\yd$ admitting the said action, then there must exist a map consisting of blowups going from $\pl\times\pl$ to $\yd$ which respects the action at each step. Now, the existence of the latter map can be checked algorithmically only using the information coming from the cycle of self intersections of $D$. Moreover, if $D$ is a cusp of length $4\leq n\leq 10$ we have the following result.

\begin{prop}\label{n10}
Let $D$ be a symmetric cusp of length $4\leq n\leq 10$. Then there always exists a surface $Y$ where $D$ sits as an anticanonical divisor equipped with an antisymplectic involution $j$ that is fixed point free on $Y\setminus D$.
\end{prop}

\begin{proof}[proof of \ref{n10}]
We begin with $(\pl\times\pl,\Delta)$ with the action defined in \ref{baseact} associated to the involution $j_0$: note that $j_0$ is such that $j_0(D_i)=D_{\sigma(i)}$, where $\sigma$ is the reflection given by $1\mapsto 3$, 2,4 are fixed.\par
If $n=4$, then the required $\yd$ is obtained from $(\pl\times\pl,\Delta)$ via a sequence of interior blowups. Moreover given the symmetries of the self intersections of  the $D_i$'s, every time we blow up at a point $p$ on $D_1$ we need to blow up at a point on the corresponding $D_{\sigma(1)}=D_3$ (and vice versa) and we can always choose this point to be $j(p)$; similarly the number of blowups needed at points lying on $D_2,D_4$ is even, thus we can always perform them in pairs at points $p,j(p)$, respectively on $D_2$ or $D_4$ avoiding the points fixed by the action. We are in the situation described in remark \ref{resact}, hence the involution $j$ lifts to the negative definite Looijenga pair $\yd$.\par
If $n>4$ then we first perform $n-4$ corner blowups to get a toric pair of the right length among those whose cycles are described in figure \ref{tp}. We observe that:
\begin{itemize}
\item [i.] The self intersections for these cycles are minimal, in the following sense: all of the $-D_i^2$ are either 1 or 2, except possibly for a pair of curves with self intersection -3 or a single curve with self intersection -4, so that any other negative definite Looijenga pair with symmetric $D$ can be obtained by a sequence of non toric blowups from one of these pairs. Indeed, recall that for a Looijenga pair to be negative definite we must have $-D_i^2\geq 2$ for all $i$ and at least one $j$ such that $-D_j^2\geq 3$. Moreover, since $D$ is symmetric, the restrictions are a bit tigher: for at least two indices $j,\sigma(j)$ we must have $-D_{j}^2\geq 3, -D_{\sigma(j)}^2\geq 3$ and if $i$ is an fixed by the reflection $\sigma$, and $-D_i^2\geq 3$, then in fact $-D_i^2\geq 4$ (it must be even).
\item [ii.] The number of nodes we need to blow up is always even and they come in pairs $\{p,j(p)\}$.
\end{itemize}
Therefore using remark \ref{resact} we can extend the action to each one of these toric pairs and thanks to the properties of the cycle of integers $(d_1,\dots,d_n)$ every $D$ of length $n=6,8,10$ sits on a smooth rational surface $Y$ that can be obtained by at least one of the toric pairs in our list through a sequence of interior blowups. We can thus repeat the argument used for $n=4$ to conclude the proof.
\end{proof}

\begin{figure}[!h]
\begin{center}
\begin{tikzpicture}[out=20,in=160,relative]
\draw (-1,0) to (1,0) to (1,2) to (1,2) to  (-1,2) to (-1,0);
\node [above] at (0,2) {0};
\node [right] at (1.2,1) {0};
\node [below] at (0,0) {0};
\node [left] at (-1.2,1) {0};
\node at (0,1) {$n=4$};
\draw (5,0) to (7,0) to (7,1) to (7,2) to (5,2) to (5,1) to  (5,0);
\node [above] at (6,2) {-2};
\node [right] at (7.2,1.5) {-1};
\node [right] at (7.2,0.5) {-1};
\node [below] at (6,0) {0};
\node [left] at (4.8,0.5) {-1};
\node [left] at (4.8,1.5) {-1};
\node at (6,1) {$n=6$};

\draw (-1,-5) to (1,-5) to (1,-4) to (1,-3) to (1,-2) to (-1,-2) to (-1,-3) to (-1,-4) to  (-1,-5);
\node [above] at (0,-2) {-2};
\node [right] at (1.2,-2.5) {-1};
\node [right] at (1.2,-3.5) {-2};
\node [right] at (1.2,-4.5) {-1};
\node [below] at (0,-5) {-2};
\node [left] at (-1.2,-4.5) {-1};
\node [left] at (-1.2,-3.5) {-2};
\node [left] at (-1.2,-2.5) {-1};
\node at (0,-3.5) {$n=8$};
\draw (5,-6) to (7,-6) to (7,-5) to (7,-4) to (7,-3) to (7,-2) to (5,-2) to (5,-3) to (5,-4) to (5,-5) to  (5,-6);
\node [above] at (6,-2) {-2};
\node [right] at (7.2,-2.5) {-3};
\node [right] at (7.2,-3.5) {-1};
\node [right] at (7.2,-4.5) {-2};
\node [right] at (7.2,-5.5) {-2};
\node [below] at (6,-6) {0};
\node [left] at (4.8,-5.5) {-2};
\node [left] at (4.8,-4.5) {-2};
\node [left] at (4.8,-3.5) {-1};
\node [left] at (4.8,-2.5) {-3};
\node at (6,-4) {\footnotesize{$n=10(i)$}};

\draw (-1,-12) to (1,-12) to (1,-11) to (1,-10) to (1,-9) to (1,-8) to (-1,-8) to (-1,-9) to (-1,-10) to (-1,-11) to  (-1,-12);
\node [above] at (0,-8) {-2};
\node [right] at (1.2,-8.5) {-1};
\node [right] at (1.2,-9.5) {-3};
\node [right] at (1.2,-10.5) {-1};
\node [right] at (1.2,-11.5) {-2};
\node [below] at (0,-12) {-2};
\node [left] at (-1.2,-11.5) {-2};
\node [left] at (-1.2,-10.5) {-1};
\node [left] at (-1.2,-9.5) {-3};
\node [left] at (-1.2,-8.5) {-1};
\node at (0,-10) {\footnotesize{$n=10(ii)$}};
\draw (5,-12) to (7,-12) to (7,-11) to (7,-10) to (7,-9) to (7,-8) to (5,-8) to (5,-9) to (5,-10) to (5,-11) to  (5,-12);
\node [above] at (6,-8) {-4};
\node [right] at (7.2,-8.5) {-1};
\node [right] at (7.2,-9.5) {-2};
\node [right] at (7.2,-10.5) {-2};
\node [right] at (7.2,-11.5) {-1};
\node [below] at (6,-12) {-2};
\node [left] at (4.8,-11.5) {-1};
\node [left] at (4.8,-10.5) {-2};
\node [left] at (4.8,-9.5) {-2};
\node [left] at (4.8,-8.5) {-1};
\node at (6,-10) {\footnotesize{$n=10(iii)$}};

\end{tikzpicture}
\end{center}
\caption{Cycles of the toric pairs}\label{tp}
\end{figure}
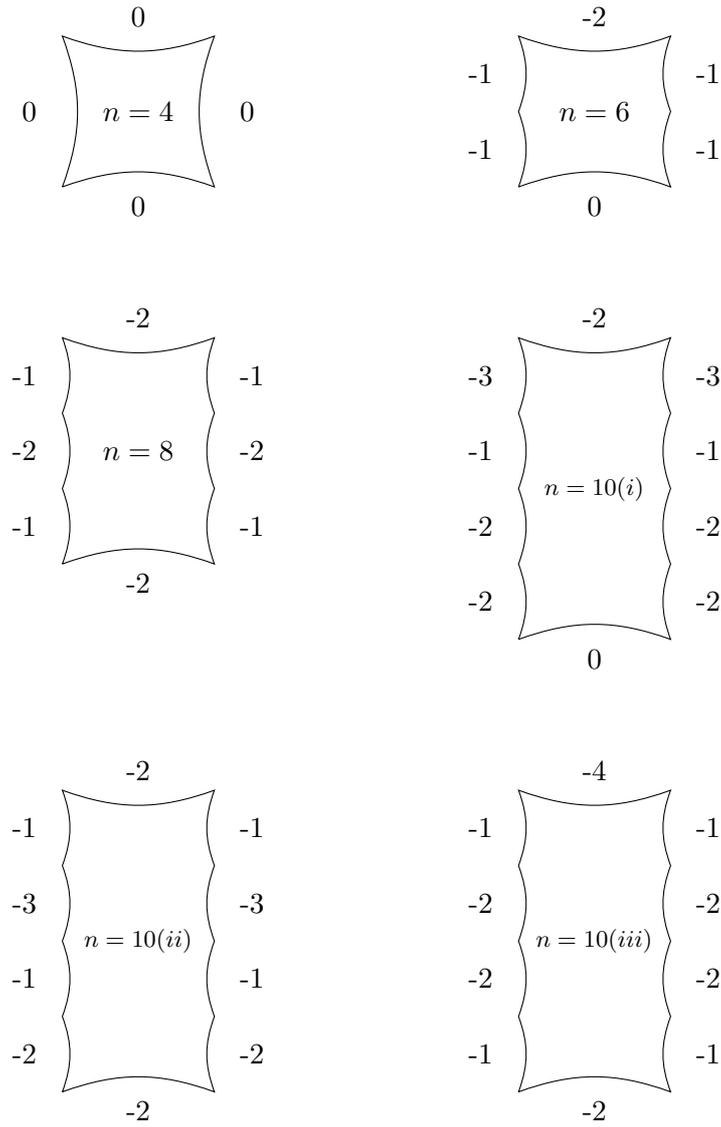

\section{Deformation theory}

Let $\px$ be a cusp singularity and $(p'\in X')$ be its dual. Let $\pi_X:\xres\rightarrow X$ and $\pi_{X'}:\widetilde{X}'\rightarrow X$ be their respective minimal resolutions with $E=\pi_X^{-1}(p)$ and $D=\pi_{X'}^{-1}(p')$. Looijenga proposed a conjecture (that now has been completely proved, see \cite{L81} and \cite{GHK15a}) which gives a sufficient and necessary condition for smoothability of cusp singularities.
\begin{theo}\label{looij}
A cusp singularity $\px$ is smoothable if and only if the dual cycle $D$ sits as an anticanonical divisor on a smooth rational surface.
\end{theo}
The aim of this section is to provide a similar result for cusp singularities that admit an antisymplectic involution that is fixed point free on $X\setminus\{p\}$ and therefore a $\zz$-action. The following conjecture, modeled on theorem \ref{looij}, would give necessary and sufficient conditions for any of these cusp singularities to be equivariantly smoothable.
\begin{conj}\label{mainc}
Let $\px$ be a cusp singularity equipped with an antisymplectic involution $\iota$ that is fixed point free on $X\setminus\{p\}$. Then $p\in X$ admits an equivariant smoothing with respect to the $\zz$-action induced by $\iota$ if and only if the dual cycle $D$ sits as an anticanonical divisor on a smooth rational surface $Y$ which admits an antisymplectic involution $j$ that is fixed point free on $X\setminus D$ and extends the one induced on $D$ by $\iota$.
\end{conj}

Even though we strongly believe that this conjecture is true, as for now, we do not have a full proof of this result, but we have a proof of its sufficiency.

\subsection{Proof of the sufficiency of Conjecture \ref{mainc}}\label{maincproof}
For the reminder of the section we will assume that the multiplicity of $\px$ is greater or equal to four. If it is equal to two, then the cusp embeds as an hypersurface in $\C^3$, so the sufficiency of \ref{mainc} can be proved by hand (cfr. proposition \ref{n10def}). We will use some of the results contained in \cite{GHK15a} to prove the following. 

Let $\px$ be a smoothable cusp singularity, $\iota$ an antisymplectic involution defined on it which is free on $X\setminus p$ and $\sigma$ the reflection induced by it on $E$. By theorem \ref{dualact} the $\zz$-action on $\px$ gives an action on the dual cusp and an associated reflection $\sigma'$ on the cycle $D$. Theorem \ref{looij} states that there exists a smooth rational surface $Y$ containing $D$ as an anticanonical divisor: suppose this Looijenga pair $\yd$ admits an antisymplectic involution $j$ which agrees with $\sigma'$ on $D$. Then the deformation family described in \cite{GHK15a} gives an equivariant smoothing of the cusp singularity $\px$.

We start by fixing some notation. Let $\bsig$ be the tropicalization of $\yd$ (cfr. \cite{GHK15a}, subsection 1.2), where $B$ is an integral affine manifold with singularities and $\Sigma$ is a decomposition of $B$ into two dimensional cones $\{\sigma_{i,i+1}\}$ generated by rays $\rho_i,\rho_{i+1}$. Let $f:Y\rightarrow Y'$ be the contraction of $D$ to the cusp singularity $q\in Y'$. The involution $j$ on $\yd$ induces an involution, namely $j^*$, on $\Pic \ Y$ as well as one on $A_1(Y,\R)$ and an involution $\theta$ on $(B,\Sigma)$.

Now, there exists (see \cite{GHK15a}, subsection 7.2) $L$, a nef divisor such that
$$\NE(Y)_{\R_{\geq 0}}\cap L^\perp=\langle D_1,\dots,D_n\rangle_{\R_{\geq 0}}$$
Note that, possibly replacing $L$ with $L'=L + j^* L$, we can always assume it is nef and invariant under the action of $\alpha$. Indeed, given that $j^*$ permutes the classes $[D_1],\dots,[D_n]$, then $\langle D_1,\dots,D_n \rangle\subset (j^* L)^\perp$ and $\NE(Y)_{\R_{\geq 0}}\cap (L+j^* L)^\perp=\langle D_1,\dots,D_n\rangle_{\R_{\geq 0}}$. Let $\sigma\subset A_1(Y,\R)$ be a strictly convex rational polyhedral cone containing $\NE(Y)$. Set $\sigma_P=\sigma\cap j^*(\sigma)$. Then $\sigma_P$ is invariant under $j^*$, it still contains $\NE(Y)$ and, possibly intersecting $\sigma_P$ with the halfspace of curves $\beta$ such that $\beta\cdot L\geq 0$, we can always assume that $\sigma_{bdy}:=\sigma_P\cap L^\perp$ is a face of $\sigma_P$ and we have that $j^*(\sigma_{bdy})=\sigma_{bdy}$. Let $P=\sigma_P\cap A_1(Y,\z)$ be the toric monoid associated to $\sigma_P$, $\mathfrak m=P\setminus 0$ and $J=P\setminus P\cap L^\perp$. We write $S=\specP$ and $S_I=\specP/I$ for any monomial ideal $I$. 

The map $j^*$ gives an involution on $S$ and $S_I$ for any $j^*$-invariant monomial ideal $I$ defined by 
\begin{equation}\label{baseinv}\alpha: z^\beta \mapsto (-1)^{\beta\cdot (D_n+D_{n/2})}z^{j^*(\beta)}\end{equation}
where $D_n, D_{n/2}$ are the two components of $D$ fixed by $j$, in particular we get
$$\alpha(z^{[D_i]})= \begin{cases}
(-1)^{D_i\cdot D_i}z^{j^*([D_i])}=z^{[D_i]} \quad \textnormal{if} \ i=n/2,n\\
-z^{j^*([D_i])}=-z^{[D_{n-i}]} \quad \textnormal{if} \ i=1,n/2-1,n/2+1,n-1\\
z^{j^*([D_i])}=z^{[D_{n-i}]} \quad \textnormal{otherwise}
\end{cases}$$

\bigskip
\noindent Note that if $n=4$ the third case does not exists and in the second case there is not a minus sign (this observation will carry out through the whole proof). The first step of the proof consists of showing that the family $f_J:X_J\rightarrow S'_J$ of \cite{GHK15a}, theorem 7.5 admits a $\zz$ action. In order to do so, let us define maps $\iota_{i,J}$, for $i=1,\dots,n$ as follows:
\begin{align*}
& \iota_{i,J}: \mathbb A^2_{x_{i-1},x_{i+1}}\times (\mathbb G_m)_{x_i}\times S_J \rightarrow \mathbb A^2_{x_{n-i-1},x_{n-i+1}}\times (\mathbb G_m)_{x_{n-i}}\times S_J \\ 
 & (x_{i-1},x_{i+1},x_i,z_J) \mapsto ((-1)^{\varepsilon(i+1)}x_{i+1},(-1)^{\varepsilon(i-1)}x_{i-1},(-1)^{\varepsilon(i)}x_i,\alpha(z_J))
\end{align*}
where $\varepsilon(i)=1$ if $i \equiv 0 \ \textnormal{mod} \  (n/2)$ and 0 otherwise and indices are meant mod $n$ within the range $1,\dots,n$ when needed. Then the image of each $U_{\rho_i,J}$ under the corresponding involution $\iota_{i,J}$ is $U_{\rho_{n-i},J}$: let $p=(x_{i-1},x_{i+1},x_i,z_J)$ be a point in $U_{\rho_i,J}$, then in coordinates we have
$$\iota_{i,J}(p)=((-1)^{\varepsilon(i+1)}x_{i+1},(-1)^{\varepsilon(i-1)}x_{i-1},(-1)^{\varepsilon(i)}x_i,\alpha(z_J))$$
which satisfies the equation $X_{n-i-1}X_{n-i+1}-z^{[D_{n-i}]}
X_{n-i}^{-D_{n-i}}=0$ defining $U_{\rho_{n-i}}$. Indeed, by substitution we get:
$$(-1)^{\varepsilon(i-1)}(-1)^{\varepsilon(i+1)}x_{i-1}x_{i+1}-\alpha(z^{[D_i]})[(-1)^{\varepsilon(i)}x_i]^{-D^2_i}=0$$
that is always a true statement: if $\varepsilon(i)=1$, then $i$ is either $n/2$ or $n$, but $(-x_i)^{-D^2_i}=x_i^{-D^2_i}$ since $-D^2_i$ is even, while if either $\varepsilon(i-1)=1$ or $\varepsilon(i+1)=1$, then $i\in\{1,n/2-1,n/2+1,n-1\}$ therefore $\alpha(z^{[D_i]})=-z^{[D_{n-i}]}$. If all the exponents are equal to 0, then the equality holds trivially, thus $\iota_{i,J}(p)\in U_{\rho_{n-i},J}$. The other containment also follows, since $\iota_{i,J}$ is an involution. As a consequence, the open analytic subsets $V_{\rho_i,J}$ (see the proof of theorem 7.5 in \cite{GHK15a} for the exact definition of the subsets $V_{\rho_i,J}$) are permuted accordingly, with $\iota_J(V_{\rho_i,J})=V_{\rho_i,J}$ if $i=n/2,n$ and $\iota_J(V_{\rho_i,J})=V_{\rho_{n-i},J}$ otherwise. Moreover the maps $\iota_{i,J}$ agree on the intersections, thus giving an involution $\iota_J$ on $\bigcup_i V_{\rho_i,J}$. Indeed, for $V_{\rho_i,J}\cap V_{\rho_{i+1},J}=(\mathbb G^2_m)_{x_i,x_{i+1}}\times S_J$ we have
$$\iota_{i,J}(x_i,x_{i+1},z_J)=((-1)^{\varepsilon(i)}x_{i},(-1)^{\varepsilon(i+1)}x_{i+1}, \alpha(z_J))=\iota_{i+1,J}(x_i,x_{i+1},z_J)$$
This gives an analytic involution $\iota_J$ on $X^\circ_J=\cup_{\rho\in\Sigma}V_{\rho,J}$ that can be extended to $X_J/S'_J$, since its fibres satisfy Serre's condition $S_2$. The analytic involution obtained through the extension to the singular locus, that we will still denote by $\iota_J$, is compatible with the one given in (\ref{baseinv}) on the base space $S'_J$. By direct computation on the charts $U_{\rho_i,J}$ one can check that this involution is fixed point free on $X^\circ_J$. Besides, the way it is defined on $x_1,\dots,x_n$ determines the way it acts on the exceptional cycle associated to the cusp singularities $s(t)\in X_{J,t}$ (cfr. lemma 7.3 in \cite{GHK15a} and remark \ref{toricdualact}). Thanks to proposition \ref{uniqueinv} the action induced by $\iota_J$ on the cusp singularities $s(t)\in X_{J,t}$ of the general fibres of $f_J$ is exactly the $\zz$ action given by hypothesis on $\px$.

\bigskip
The next step is to consider the thickening of the cusp family as it is presented in \cite{GHK15a}, theorem 7.7.
The involution $\iota_J$ defined on the family $f_J:X_J\rightarrow S'_J$ extends to its thickening $f_I:X_I\rightarrow S'_I$. Indeed the scattering diagram $\mathfrak D$ is invariant with respect to the involution, in the following sense: given a pair $(\mathfrak d, f_{\mathfrak d})\in \mathfrak D$ then the action induced by $\iota_J$ on it is given by $\iota_J\cdot(\mathfrak d, f_{\mathfrak d})=(\theta(\mathfrak d),f_{\mathfrak d}\circ \iota_J)$. We claim that the latter is still contained in $\mathfrak D$.
Indeed, start with considering the functions associated to the rays in $\Sigma$: each $f_{\rho_i}$ is a truncated version of
$$\exp\left [ \sum_\beta (\beta\cdot D_i) N_\beta z^{\beta}x_i^{-\beta\cdot D_i}\right ]$$
where the sum runs over all classes $\beta\in A_1(Y,\z)$ satisfying the property
\begin{equation}\label{betadef}\beta\cdot D_i\neq 0 \ \textnormal{and} \  \beta\cdot D_j=0\  \textnormal{for all} \  j\neq i \end{equation}
and $N_\beta$ is defined in \cite{GHK15a}, definition 3.1. Thus, when applying $\iota_{i,J}$, for each admissible $\beta$ we get
\begin{equation}\label{dfunct}(\beta\cdot D_i) N_\beta (-1)^{\beta\cdot (D_{n/2}+D_n)}z^{j^*(\beta)}(-1)^{\varepsilon(i)\beta\cdot D_i}x_{n-i}^{-\beta\cdot D_i}\end{equation}
If $i=n/2,n$, then $(-1)^{\beta\cdot (D_{n/2}+D_n)}$ becomes $(-1)^{\beta\cdot D_i}$ and $\varepsilon(i)=1$, giving $(-1)^{\varepsilon(i)\beta\cdot D_i}=(-1)^{\beta\cdot D_i}$. Otherwise, $\beta\cdot (D_{n/2}+D_n)=0$ and $\varepsilon(i)=0$, thus we obtain
$$(\ref{dfunct})=\begin{cases}
(\beta\cdot D_i) N_\beta z^{j^*(\beta)}x_{i}^{-\beta\cdot D_i} \quad \textnormal{if} \ i=n/2,n\\ 
(\beta\cdot D_i) N_\beta z^{j^*(\beta)}x_{n-i}^{-\beta\cdot D_i} \quad \textnormal{otherwise}
\end{cases}$$
Since $j$ is an involution, if $\beta$ is a class in $A_1(Y,\z)$, then $j^*\beta\cdot j^*C=\beta\cdot C$ for any curve $C$ in $Y$ intersecting $\beta$ properly. Therefore if $\beta$ satisfies (\ref{betadef}) for some $i$, then $j^*\beta$ is such that (\ref{betadef}) is true with $i$ replaced by $n-i$, and indices considered mod $n$ as usual. Moreover, $N_{j^*\beta}=N_\beta$ because the definition of $N_\beta$ is determined by a moduli space that in turns only depends on the isomorphism class of $(\yd,\beta)$. Besides, f $\gamma$ is a class in $A^1(Y,\z)$ satisfying (\ref{betadef}) for $n-i$, then $\gamma=j^*(\beta)$ for some $\beta$ satisfying (\ref{betadef}) as well, hence 
$$\sum_\beta (\beta\cdot D_i) N_\beta z^{j^*(\beta)}x_{n-i}^{-\beta\cdot D_i}=\sum_\gamma (\gamma\cdot D_i) N_\gamma z^{\gamma}x_{n-i}^{-\gamma\cdot D_i}$$
where the sum runs over the appropriate classes $\beta$ and $\gamma$ respectively. Therefore $f_{\rho_i}\circ \iota_{i,J}=f_{\rho_i}$ for $n/2,n$ and  $f_{\rho_i}\circ \iota_{i,J}=f_{\rho_{n-i}}$ otherwise. Now suppose $f_{\mathfrak d}$ is the function relative to any other ray $\mathfrak d$ of rational slope, with $\mathfrak d$ contained in the cone of $\Sigma$ generated by $\rho_i,\rho_{i+1}$. Then $f_{\mathfrak d}$ is the truncated version of 
$$\exp\left [ \sum_\beta \kappa_\beta N_\beta z^{\beta}x_i^{-a\kappa_\beta}x_{i+1}^{-b\kappa_\beta}\right ]$$
with $\beta$ defined as usual and $a,b$ chosen to satisfy $\mathfrak d=\R_{\geq 0}(a\kappa v_i+b\kappa v_{i+1})$, with $\rho_i=\R_{\geq 0}v_i,\rho_{i+1}=\R_{\geq 0}v_{i+1}$ and $\kappa_\beta$ the positive integer such that $\beta\cdot D_{i}=a\kappa_\beta$ and $\beta\cdot D_{i+1}=b\kappa_\beta$. Therefore, applying $\iota_{i,J}$, for each admissible $\beta$ we obtain
\begin{equation}\label{dfunctrat} \kappa_\beta N_\beta (-1)^{\beta\cdot (D_{n/2}+D_n)}z^{j^*(\beta)}g_i\end{equation}
with
$$g_i=\begin{cases}
(-1)^{\varepsilon(i)a\kappa_\beta+\varepsilon(i+1)b\kappa_\beta}x_{n-i}^{-a\kappa_\beta}x_{n-i-1}^{-b\kappa_\beta} \quad \textnormal{if} \ i\neq n/2,n \ and \ i+1\neq n/2,n \\
(-1)^{\varepsilon(i)a\kappa_\beta+\varepsilon(i+1)b\kappa_\beta}x_{i}^{-a\kappa_\beta}x_{i-1}^{-b\kappa_\beta} \quad \textnormal{if} \ i= n/2,n  \\
(-1)^{\varepsilon(i)a\kappa_\beta+\varepsilon(i+1)b\kappa_\beta}x_{n-i}^{-a\kappa_\beta}x_{i+1}^{-b\kappa_\beta} \quad \textnormal{if} \ i+1= n/2,n \\
\end{cases}$$
If $i=n/2$ or $i=n$, then $\varepsilon(i)=1$ while $\varepsilon(i+1)=0$, therefore reasoning as we did above, we get $(-1)^{\beta\cdot D_{i}}(-1)^{a\kappa_\beta}=1$, since $\beta\cdot D_{i}=a\kappa_\beta$. Similarly, all negative signs cancel if $i+1=n/2,n$. Thus we get
$$(\ref{dfunctrat})=
\begin{cases}
\kappa_\beta N_\beta z^{j^*(\beta)} x_{n-i}^{-a\kappa_\beta}x_{n-i-1}^{-b\kappa_\beta} \quad \textnormal{if} \ i\neq n/2,n \ and \ i+1\neq n/2,n \\
\kappa_\beta N_\beta z^{j^*(\beta)}x_{i}^{-a\kappa_\beta}x_{i-1}^{-b\kappa_\beta} \quad \textnormal{if} \ i= n/2,n  \\
\kappa_\beta N_\beta z^{j^*(\beta)}x_{n-i}^{-a\kappa_\beta}x_{i+1}^{-b\kappa_\beta} \quad \textnormal{if} \ i+1= n/2,n \\
\end{cases}$$
Arguing as before we can now show that $f_{\mathfrak d}\circ \iota_{i,J}=f_{\theta(\mathfrak d)}$ as needed: this concludes the proof that the scattering diagram $\mathfrak D$ is $\theta$-invariant.
The hypersurfaces $U_{\rho_i,I}$ are defined by the equations 
$$x_{i-1}x_{i+1}-z^{[D_i]}x_i^{-D^2_i}f_{\rho_i}=0$$
thus the analysis above implies that the maps $\iota_{\rho_i,J}$ extend to each $V_{i,I}\subset U_{\rho_i,I}$ and respect all the gluing isomorphisms, giving a new involution $\iota_I$ on $X_I/S'_I$.

\bigskip
Finally, we need to describe the subspace $S'\subset S$ of points fixed by the involution $\alpha$. In order to do this let us recap our notation. Let $T$ be the algebraic torus contained in the affine toric variety $S$ and recall that $S=\textnormal{Spec} \ \C[\sigma_P \cap M]$, where $M= \textnormal{A}_1(Y)$ and the dual lattice $N=\textnormal{Pic}\ Y$. 
Since by assumption the involution $j$ defined on the pair $\yd$ is fixed point free away from $D$, then we can use theorem \ref{desact} to get a precise description of the pullback map $j^*$ on the Picard group of $Y$. Indeed, the theorem states that there exists a sequence of maps 
$$\yd \xrightarrow{\psi_1}(Y_1,D_1)\xrightarrow{\psi_2} \cdots \xrightarrow{\psi_m} (\pl\times\pl,\Delta)$$
where each $\psi_i$ corresponds to the blowup of two points on the anticanonical divisor of the pair $(Y_i,D_i)$ which belong to the same orbit with respect to the action induced on $(Y_i,D_i)$ by the original involution defined on $\yd$. We may always assume that the maps $\psi_1,\dots,\psi_t$ are pairs of interior blowups, while the remaining ones are pairs of toric blowups, thus implying that $(Y_t,D_t)$ is an equivariant toric model for $\yd$ (see remark \ref{equivtoricmod}). This sequence of maps induces on $(\pl\times\pl,\Delta)$ the involution $j_0$ defined as $(z,w)\mapsto (z^{-1},-w)$ (cfr. remark \ref{baseact}). Observe that the pullback of $j_0$ acts trivially on the Picard group of $\pl\times\pl$ and that a basis for $\textnormal{Pic}(Y)$ is given by the pullbacks $F_1,F_2$ of the two generators of $\textnormal{Pic}(\pl\times\pl)$, by the classes of the divisors $D_1,\dots D_n$ excluding the strict transforms of the four boundary divisors of $\pl\times\pl$ and by those of the exceptional divisors $\mathcal E=\{E_{i,j},E'_{i,\sigma(j)}\}_{i=1,\dots,t}$, of the maps $\psi_1,\dots,\psi_t$ or, to be more precise, of their strict transforms in $Y$. Here the index $j$ refers to the divisor $D_j$ intersected by $E_{i,j}$. Since the involution $j$ fixes $F_1,F_2$, maps each $E_{i,j}$ to $E'_{i,\sigma(j)}$ (and viceversa), and each $D_i$ to $D_{\sigma(i)}$ (and viceversa) the involution induced by it on the Picard group of $Y$ corresponds to the block diagonal matrix

$$\begin{pmatrix}1 & 0 &   &    &           &   & 0\\
					   0 & 1 &   &    &           &   &  \\
					     &    & 0 & 1 &           &   &  \\
					     &    & 1 & 0 &           &   &  \\
					     &    &    &    & \ddots &   &  \\
					     &    &    &    &          & 0 & 1\\
					   0 &    &    &    &         & 1 & 0\end{pmatrix}$$
Furthermore, let  $(z_1,z_2,w_1,\dots,w_{2k},u_1,\dots, u_{2l})$ be the coordinates of the torus $T$ given by this basis (identifying $\textnormal{Pic }Y$ and $A_1(Y)$, the characters of $T$), arranged so that $z_1,z_2$ correspond to $F_1,F_2$, the pairs $(w_1,w_2),\dots,(w_{2k-1},w_{2k})$ are of the form $z^{[E_{i,j}]},z^{[E'_{i,\sigma(j)}]}$ with $j=n/2,n$ or $z^{[D_j]},z^{[D_{\sigma(j)}]}$ where $D_j$ meet $D_{n/2},D_n$ and the pairs  $(u_1,u_2),\dots,(u_{2l-1},u_{2l})$ correspond to the divisors in $E_{i,j},E'_{i,\sigma(j)}$ with $j\neq n/2,n$ or to divisors $D_j,D_{\sigma(j)}$ which do not intersect $D_{n/2}$, $D_n$. If we call still $\alpha$ (as in \ref{baseinv}) the involution defined on $T$ by the formula 
$$\alpha: z^\beta \mapsto (-1)^{\beta\cdot ([D_n]+[D_{n/2}])}z^{j^*(\beta)}$$
then in coordinates we get that this involution maps $(z_1,z_2,w_1,\dots,w_{2k},u_1,\dots,u_{2l})$ to $(z_1,z_2,-w_2,-w_1,\dots,-w_{2k},-w_{2k-1},u_2,u_1,\dots,u_{2l},u_{2l-1})$. Therefore the fixed locus for $\alpha$ in $T$ is described by the equations $w_{2i}=-w_{2i-1}$ for $i=1,\dots,k$ and $u_{2i}=u_{2i-1}$ for $i=1,\dots,l$ and it is the translate of a subtorus $T'\subset T$ of dimension $2+k+l$ by a point of type $(\pm1,\dots,\pm1)$, where $T'=N'\otimes \C^*$ and $N'\subset N$ is the sublattice in $N$ fixed by the involution. Finally let $\nu=\sigma_P^*$ be the cone in $N_{\R}$ associated to $S$ and $\nu'\subset \nu$ be defined as $\nu\cap N'_{\R}$ and similarly $\tau'=\tau\cap N'_{\R}$, where $\tau=\nu\cap\sigma_{bdy}^{\perp}$, the face of $\nu$ corresponding to the face $\sigma_{bdy}$ of $\sigma_P$. Then, by standard toric geometry arguments, the closure of $T'$ in $S$ is the toric variety $S'$ which corresponds to $\nu'$. $S'$ meets the interior of the stratum $Z$ corresponding to $\tau$ when $N'_{\R}$ meets the relative interior of $\tau$ and in this case $\overline{S'\cap \textnormal{Int}(Z)}$ automatically contains $0\in Z$. Observe that $\tau$ and $N'_{\R}$ are invariant under the $\zz$-action induced by $j^*$. We claim that $N'_{\R}$ intersects the interior of $\tau$ non trivially. Indeed, let $V:=\textnormal{span}_\R\langle\tau\rangle$, then $\tau$ is full dimensional in $V$ and this vector space decomposes in eigenspaces for $j^*$ as $V^+\oplus V^-$, with corresponding eigenvalues $1,-1$. Let $p:V\rightarrow V^-$ be the projection onto the second eigenspace and consider the restriction of the involution to $V^-$: here $j^*$ acts as $-$Id and $p(N'_\R\cap V)=0$. The image $\tau^-:=p(\tau)$ under $p$ of $\tau$ is full dimensional and invariant under the action of $-$Id. It follows that 0 must be contained in the interior of $\tau^-$, and therefore $N'_\R$ must intersect the interior of $\tau$ non trivially as well.

In order to be sure that $S'$ provides a smoothing of the cusp singularity dual to $D$, thanks to definition 4.2 and lemma 7.15 in \cite{GHK15a}, it suffices to consider the Gross-Siebert locus $\hat S$ contained in $S'$ (cfr. definition 3.14 in \cite{GHK15a}). Indeed, thanks to lemma 7.15 a smoothing of the cusp exists if and only if there exists a smoothing in a formal neighborhood of the Gross-Siebert locus. The map $\pi=\psi_t\circ\dots\circ\psi_1$ determines a face of $\textnormal{Nef} \ Y$, namely $\pi^*(\textnormal{Nef} \ Y_t)$, or equivalently a face of $\textnormal{NE}\ Y$. We may always assume that there is a corresponding face $F$ of $\nu$ that is $\zz$-invariant (because $\yd$ has an equivariant toric model, see remark \ref{equivtoricmod}). On $\hat S$ we have coordinates $\{z^{[E_{i,j}]},z^{[E'_{i,\sigma(j)}]}\}_{i=1,\dots,t}$ and $j\in\{1,\dots,n\}$, while the fixed locus in $\hat S$ is described by the equations 
\begin{equation}\begin{cases}\label{eqgs}
z^{[E_{i,j}]}=-z^{[E'_{i,\sigma(j)}]} \qquad \textnormal{if} \ j=n/2,n \\
z^{[E_{i,j}]}=z^{[E'_{i,\sigma(j)}]} \qquad \textnormal{otherwise}
\end{cases}\end{equation}
therefore it has coordinates $z^{[E_{i,j}]}$ with $i=1,\dots,t$ and $j\in\{1,\dots,n\}$. In general if $\{z_{i,j}\}$ are the coordinates of the Gross-Siebert locus corresponding to the exceptional divisors meeting $D_j$, for the smoothness argument we must have $z_{i,j}\neq z_{k,j}$ for all $j$ and for all $i\neq k$. In our case, because of (\ref{eqgs}) and the way $\sigma$ is defined, this reduces to check that
$z^{[E_{i,j}]}\neq{z^{[E'_{k,j}]}}$ for $j=n/2,n$ and for all $i\neq k$. This is always true, because of the first equation in (\ref{eqgs}). Now if $I$ is an $\alpha$-invariant monomial ideal, then we can consider $S'_I\subset S_I$ and, by restriction, we get an equivariant family $f':X'_I\rightarrow S'_I$. From this family, using the techniques of theorem 7.13 in \cite{GHK15a} we finally obtain an equivariant smoothing of the cusp singularity $\px$.

\subsection{Cusps with embedding dimension $n\leq 12$}
We describe what is known about the conjecture \ref{mainc} for cusp singularities of embedding dimension $n\leq 12$. It can be proved that for $n\leq 10$ it is always possible to find an equivariant smoothing of the cusp $\px$.
\begin{prop}\label{n10def}
Every germ of a symmetric cusp singularity of embedding dimension $n\leq 10$ is equivariantly smoothable.
\end{prop}
\begin{proof}
It can be checked that symmetric cusp singularities of multiplicity 2 are always equivariantly smoothable, since they embed in $\mathbb A^3$ as hypersurfaces and have an explicit description of their smoothings. Indeed, the equation of a cusp $\px$ of multiplicity 2 is given by $(z^2+x^p+y^q+xyz=0)\subset \mathbb A^3$ with $\frac{1}{2}+\frac{1}{p}+\frac{1}{q}<1$; we can change coordinates so that it is given by $(z^2+x^p+y^q-\frac{1}{4}x^2y^2=0)\subset \mathbb A^3$. Note that since the cusp is symmetric, then $p,q$ have to be even. Now, the involution on $\px$ is given by $(x,y,z)\mapsto(-x,-y,-z)$: this is the right involution since it preserves the cusp and it is antisymplectic, as it can be checked on the minimal resolution.  Finally, a smoothing of $\px$ can be described as the one parameter family $(z^2+x^p+y^q-\frac{1}{4}x^2y^2+t=0)\subset \mathbb A^3\times \mathbb A^1_t$. This family gives an equivariant smoothing with respect to the extended involution given by $(x,y,z,t)\mapsto(-x,-y,-z,t)$. Now suppose $\px$ is a cusp singularity of embedding dimension $4\leq n \leq 10$. Theorem \ref{dualact} gives an involution $j$ on the dual cusp $(p'\in X')$ acting freely on $X'\setminus \{p'\}$ that induces a $\zz$-action on its exceptional cycle $D$. Thus the dual cusp $D$ is symmetric and by proposition \ref{n10} there exists a rational surface $Y$ on which $D$ sits as an anticanonical divisor together with an antisymplectic involution $j$ that is fixed point free away from $D$ and we can use the sufficient condition of conjecture \ref{mainc} proved in subsection \ref{maincproof} to find an equivariant smoothing of $\px$.
\end{proof}

\begin{table}[t!]
\begin{center}
\begin{tabular}{ c|c } 

 Cusp singularity & Dual cycle of integers \\ \hline
  $(3,10,3,4)$ & $(3,3,2,2,2,2,2,2,2,3,3,2)$ \\ 
  $(3,8,3,6)$ & $(2,3,3,2,2,2,2,2,3,3,2,2)$ \\
  $(4,8,4,4)$ & $(3,2,3,2,2,2,2,2,3,2,3,2)$ \\
  $(6,4,6,4)$ & $(3,2,2,2,3,2,3,2,2,2,3,2)$ \\
  $(12,3,2,3)$ & $(3,2,2,2,2,2,2,2,2,2,3,4)$ \\
  $(10,4,2,4)$ & $(2,3,2,2,2,2,2,2,2,3,2,4)$ \\
  $(6,2,6,6)$ & $(2,2,2,3,2,2,2,3,2,2,2,4)$ \\
  $(4,7,2,7)$ & $(2,2,2,2,3,2,3,2,2,2,2,4)$ \\
  $(3,3,8,3,3,4)$ & $(3,3,3,2,2,2,2,2,3,3,3,2)$ \\
  $(3,3,6,3,3,6)$ & $(2,3,3,3,2,2,2,3,3,3,2,2)$ \\
  $(3,3,2,3,3,10)$ & $(3,3,2,2,2,2,2,2,2,3,3,4)$ \\
  $(6,3,2,3,6,4)$ & $(3,2,2,2,3,2,3,2,2,2,3,4)$ \\

\end{tabular}
\end{center}\vspace{0.4cm}\caption{Cusp singularities that do not satisfy the statement of proposition \ref{n12smooth}}\label{n12cusps}
\end{table}

To find an example of a cusp that is equipped with an antisymplectic involution $\iota$ but does not admit an equivariant smoothing, we have to look among symmetric cusps with embedding dimension at least equal to 12. In fact we conjecture that $n=12$ is big enough. More precisely, the following result can be proved.

\begin{prop}\label{n12smooth}
All smoothable symmetric cusp singularities $\px$ of embedding dimension $n=12$ are $\zz$-equivariantly smoothable, except possibly for the ones listed below in table \ref{n12cusps}. These cusps correspond to Looijenga pairs that do not admit an antisymplectic involution which is free away from the anticanonical cycle.
\end{prop}

\begin{proof}
Similarly to the proof of proposition \ref{n10}, an antisymplectic involution can be constructed for a list of minimal toric Looijenga pairs of length equal to 12, given in table \ref{table12}. Recall that they are minimal in the sense that each component has self intersection as small as possible in absolute value, while still being toric and of the right length: in this case it means that they are at most equal to 6 in absolute value. Therefore the symmetric cusp singularities corresponding to anticanonical pairs that can be obtained from the ones in table \ref{table12} through interior blowups are equivariantly smoothable thanks to the sufficient condition of conjecture \ref{mainc} proved in section \ref{maincproof}, as we observed in proposition \ref{n10def}.

We checked by exhausting all possible cases that it is not possible to construct  Looijenga pairs $\yd$, where $D$ has the cycles of integers listed in table \ref{n12cusps}, by starting from $\pl\times\pl$ and subsequently performing pairs of toric and interior blowups in a symmetric way. By proposition \ref{desact} this implies that there does not exists a Looijenga pair $\yd$ equipped with an antisymplectic involution that is fixed point free away from the anticanonical divisor, for any $D$ having cycle of integers listed in table \ref{n12cusps}. Therefore we cannot conclude that the correspondent symmetric cusps are equivariantly smoothable.

\begin{table}[h]
\begin{center}
\begin{tabular}{c} 

 Toric Looijenga pairs  \\ \hline
  $(1,2,2,2,1,2,1,2,2,2,1,6)$ \\
  $(1,2,2,2,1,4,1,2,2,2,1,4)$ \\
  $(4,1,2,2,2,0,2,2,2,1,4,2)$ \\
  $(2,2,1,4,1,2,1,4,1,2,2,2)$ \\
  $(3,2,1,3,2,0,2,3,1,2,3,2)$ \\
  $(3,1,3,1,3,0,3,1,3,1,3,2)$ \\
  $(1,3,1,3,1,2,1,3,1,3,1,4)$ \\
  $(2,1,3,2,1,2,1,2,3,1,2,4)$ \\
\end{tabular}
\end{center}\vspace{0.4cm}\caption{Toric Looijenga pairs used in the proof of proposition \ref{n12smooth}}\label{table12}
\end{table}
\end{proof}

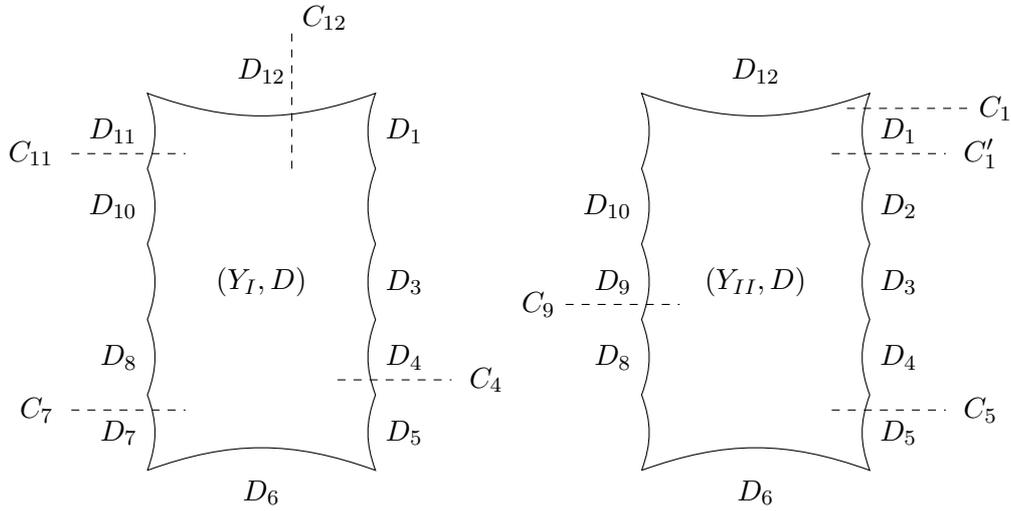
\begin{figure}[h!]
\begin{center}
\begin{tikzpicture}[out=20,in=160,relative]
\draw (-1.5,0) to (1.5,0) to (1.5,1) to (1.5,2) to (1.5,3) to (1.5,4) to (1.5,5) to (-1.5,5) to (-1.5,4) to (-1.5,3) to (-1.5,2) to (-1.5,1) to  (-1.5,0);
\draw[dashed] (0.4,4) -- (0.4,5.8);
\draw[dashed] (-2.5,4.2) -- (-1,4.2);
\draw[dashed] (-2.5,0.8) -- (-1,0.8);
\draw[dashed] (1,1.2) -- (2.5,1.2);
\node [below] at (0,0) {$D_6$};
\node [right] at (1.5,0.5) {$D_5$};
\node [right] at (1.5,1.5) {$D_4$};
\node [right] at (1.5,2.5) {$D_3$};
\node [right] at (1.5,4.5) {$D_1$};
\node [above] at (0,5) {$D_{12}$};
\node [left] at (-1.5,4.5) {$D_{11}$};
\node [left] at (-1.5,3.5) {$D_{10}$};
\node [left] at (-1.5,1.5) {$D_8$};
\node [left] at (-1.5,0.5) {$D_7$};
\node at (0,2.5) {$(Y_{I},D)$};
\node [right] at (0.4,6) {$C_{12}$};
\node [right] at (2.6,1.2) {$C_4$};
\node [left] at (-2.6,0.8) {$C_7$};
\node [left] at (-2.6,4.2) {$C_{11}$};

\draw (5,0) to (8,0) to (8,1) to (8,2) to (8,3) to (8,4) to (8,5) to (5,5) to (5,4) to (5,3) to (5,2) to (5,1) to  (5,0);
\draw[dashed] (4,2.2) -- (5.5,2.2);
\draw[dashed] (7.5,4.2) -- (9,4.2);
\draw[dashed] (7.7,4.8) -- (9.3,4.8);
\draw[dashed] (7.5,0.8) -- (9,0.8);
\node [below] at (6.5,0) {$D_6$};
\node [right] at (8,0.5) {$D_5$};
\node [right] at (8,1.5) {$D_4$};
\node [right] at (8,2.5) {$D_3$};
\node [right] at (8,3.5) {$D_2$};
\node [right] at (8,4.5) {$D_1$};
\node [above] at (6.5,5) {$D_{12}$};
\node [left] at (5,3.5) {$D_{10}$};
\node [left] at (5,2.5) {$D_9$};
\node [left] at (5,1.5) {$D_8$};
\node at (6.5,2.5) {$(Y_{II},D)$};
\node [right] at (9.3,4.8) {$C_1$};
\node [right] at (9.1,4.2) {$C'_1$};
\node [right] at (9.1,0.8) {$C_5$};
\node [left] at (4,2.2) {$C_9$};
\end{tikzpicture}
\end{center}
\caption{Schematic description of the two possible Looijenga pairs for $D$}\label{n12}
\end{figure}

\begin{ex}\label{n12inv}
Among the cusps listed in proposition \ref{n12smooth} there is the one with associated cycle of integers $(3,10,3,4)$: its dual cusp $D$ corresponds to the exceptional cycle with self intersections $(3,3,2,2,2,2,2,2,2,3,3,2)$.
By inspection it can be proved that, up to isomorphism, there are two anticanonical pairs with cycle $D$ (see figure \ref{n12}: here the dotted lines represent $(-1)$-curves) and that for both of them there does not exist a map to $\pl\times\pl$ which allows us to lift the involution $j$ defined on $\pl\times\pl$ as we did for cycles of length $n\leq 10$.
\end{ex}

\begin{conj}
The cusp singularities listed in proposition \ref{n12smooth} do not admit an equivariant smoothing.
\end{conj}

\subsection{Equivariant smoothings of simple elliptic singularities}\label{ellsect}
Let us describe what happens if, instead of cusp singularities, we consider the case of simple elliptic singularities, namely cones over elliptic curves. This is an interesting case, since these singularities and their deformation space have been studied extensively (see for example \cite{P74} and \cite{M82}) and can be described quite explicitly. 

\bigskip
For simple elliptic singularities $p\in C(E)$, where $C(E)$ is the cone over a smooth elliptic curve $E$ of degree $d$ smaller than eight, there exists essentially one smoothing component with associated Milnor fiber given by $M=S\setminus E$, where $S$ is the del Pezzo surface of corresponding degree $d$. To be more precise in these cases, a smoothing family can be obtained as follows. Let $S$ be a del Pezzo surface of degree $d$ and consider the projective closure of the affine cone over this surface $\overline{C(S)}\subset \mathbb P^{d+1}$. Let $\mathcal H_t$, with $t\in \mathbb A^1_t$ be a family of hyperplanes in $\mathbb P^n$ such that $p \in \mathcal H_t$ if and only if $t=0$ and let us consider $\mathcal X_t:=\mathcal H_t\cap \overline{C(S)}$. Then $\mathcal X\rightarrow \mathbb A^1_t$ is a smoothing family for $p\in C(E)$ with $\mathcal X_0\cong \overline{C(E)}$ and $\mathcal X_t\cong S$ for $t\neq 0$.
On the other hand, given a cone over an elliptic curve $E$ of degree eight, its deformation space is isomorphic to $(\bigcup_{i=1}^4\mathbb A^1\times \mathbb A^2) \cup C(\mathcal E))$, where $C(\mathcal E)$ is the cone over the universal elliptic curve $\mathcal E\rightarrow \mathbb A^1$. Each plane $\mathbb A^2$ and the cone $C(E)$ give a smoothing component for the singularity, thus implying that every simple elliptic singularity of degree 8 is smoothable, and they are distinguished by the associated Milnor fibre. The Milnor fibre $M_i$ corresponding to $\bigcup_{i=1}^4 \mathbb A^1\times \mathbb A^2$ is isomorphic to $\pl\times\pl\setminus E$, where $E$ again is the elliptic curve we start with. While the other smoothing component has Milnor fiber given by $M_{ii}=\mathbb F_1\setminus E$.

\bigskip
Moreover, let $M$ be the Milnor fibre of a smoothing of a simple elliptic singularity $p\in C(E)$ of degree $d$. Then $M$ is mirror (cfr. \cite{HK20} and \cite{AKO06}) to the surface $U_d=Y_{(d)}\setminus D_{(d)}$, where $(Y_{(d)},D_{(d)})$ is a semidefinite negative pair with $D_{(d)}$ a cycle of $d$  $(-2)$-curves. The two Milnor fibres $M_i,M_{ii}$, which correspond to the two different smoothing components of a simple elliptic singularity of degree eight are mirror to the surfaces $U_i=Y_i\setminus D$ and $U_{ii}=Y_{ii}\setminus D$, where $(Y_i,D)$ and $(Y_{ii},D)$ are semidefinite Looijenga pairs and $D$ is a cycle of eight rational curves of self intersection -2. They can be obtained explicitly from the toric pairs $(T_{i},G_{i})$, whose divisor $G_{i}$ has cycle of integers $(1,2,1,2,1,2,1,2)$, and $(T_{ii},G_{ii})$, that is associated to $(1,2,1,2,2,1,2,1)$, respectively through four interior blowups on the (-1)-curves contained in the toric boundaries (see figure \ref{toricfans} for a description of the toric fans of $T_{i}$ and $T_{ii}$). We observe that $M_i$ and $M_{ii}$ are not diffeomorphic, since they are diffeomorphic to the open manifolds $U_i,U_{ii}$ and these manifolds have have different fundamental groups. More precisely, $\pi_1(U_i)=\zz$ while $\pi_1(U_{ii})=0$. This can be seen recalling that if $\yd$ is obtained from a toric surface through a sequence of interior blowups, then $\pi_1(Y\setminus D)=N/\langle v_1,\dots,v_p\rangle$, where $N$ is the lattice containing the fan of the toric variety and $v_1,\dots,v_p\in N$ are the primitive vectors corresponding to the curves of the toric boundary where the blowups are performed. Now, the cones of the fan of the toric surface $T_{ii}$ are generated by the rays associated to $\{e_1, 2e_1+e_2, e_1+e_2, e_2, -e_1+e_2, -e_1, -e_1-e_2, -e_2\}$ in $N_\R$. Therefore, the set of vectors $\{2e_1+e_2, -e_1+e_2, -e_1-e_2, -e_2\}$ corresponding to the four (-1)-curves where the blowups are performed (in figure \ref{toricfans} they are marked by a red dot) contains a basis for the lattice $N$ and as a consequence the fundamental group of $U_{ii}$ is trivial. The cones of the fan of the toric surface $T_{i}$ instead are generated by the rays given by the vectors $\{e_1, e_1+e_2, e_2, -e_1+e_2, -e_1, -e_1-e_2, -e_2, e_1-e_2\}$, hence the set of vectors $\{e_1+e_2, -e_1+e_2, -e_1-e_2, e_1-e_2\}$ corresponding to the four blowups share some linear relations.  Thus $\pi_1(U_i)=\zz$.

\begin{figure}[h!]
\begin{center}
\begin{tikzpicture}[out=20,in=160,relative]
 
\draw[->] (-3,0) -- (-1.5,0);
\draw[->] (-3,0) -- (-1.5,1.5);
\draw[->] (-3,0) -- (-3,1.5);
\draw[->] (-3,0) -- (-4.5,1.5);
\draw[->] (-3,0) -- (-4.5,0);
\draw[->] (-3,0) -- (-4.5,-1.5);
\draw[->] (-3,0) -- (-3,-1.5);
\draw[->] (-3,0) -- (-1.5,-1.5);

\node [right] at (-1.5,1.5) {\textcolor{red}{$\bullet$}};
\node [left] at (-4.5,1.5) {\textcolor{red}{$\bullet$}};
\node [left] at (-4.5,-1.5) {\textcolor{red}{$\bullet$}};
\node [right] at (-1.5,-1.5) {\textcolor{red}{$\bullet$}};
\node [right] at (-1.5,0) {$e_1$};
\node [above] at (-3,1.5) {$e_2$};

\draw[->] (3,0) -- (1.5,0);
\draw[->] (3,0) -- (1.5,1.5);
\draw[->] (3,0) -- (3,1.5);
\draw[->] (3,0) -- (4.5,1.5);
\draw[->] (3,0) -- (4.5,0);
\draw[->] (3,0) -- (6,1.5);
\draw[->] (3,0) -- (3,-1.5);
\draw[->] (3,0) -- (1.5,-1.5);

\node [right] at (6,1.5) {\textcolor{red}{$\bullet$}};
\node [below] at (3,-1.5) {\textcolor{red}{$\bullet$}};
\node [left] at (1.5,1.5) {\textcolor{red}{$\bullet$}};
\node [left] at (1.5,-1.5) {\textcolor{red}{$\bullet$}};
\node[right] at (4.5,0) {$e_1$};
\node[above] at (3,1.5) {$e_2$};

\end{tikzpicture}
\end{center}
\caption{The fans of the toric surfaces $T_{i}$ and $T_{ii}$}\label{toricfans}
\end{figure}
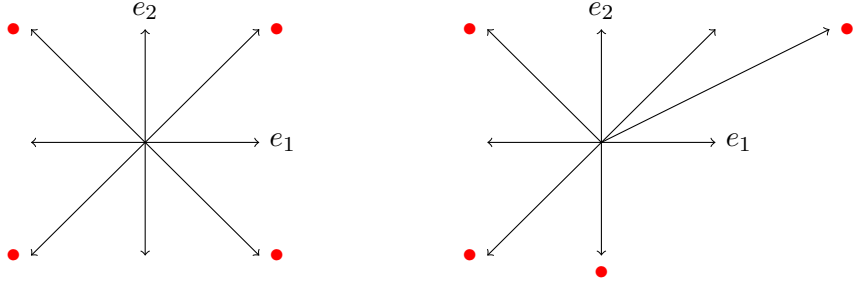

\bigskip
Let now $\iota$ be an involution on $X=C(E)$ for a given smooth elliptic curve $E$ of even degree $d\leq 8$, inducing the hyperelliptic involution on $E$ and acting by $(-1)$ on the fixed fibres. We can ask whether the existence of an equivariant smoothing  for $p\in X$ corresponds to the existence of an involution on the associated Looijenga pair $(Y_{(d)},D_{(d)})$ or, in the case $d=8$, on either $(Y_i,D)$ or $(Y_{ii},D)$, which is free on the complement of $D$ and acts as a reflection on $D$, in the spirit of the main conjecture \ref{mainc} for cusp singularities stated previously in this chapter. Let us start by stating a result on the existence of $\zz$-equivariant smoothings.

\begin{theo}
Let $E$ be a smooth elliptic curve of even degree $d\leq 8$ and $\iota$ an involution defined as above. Then there always exists a $\zz$-equivariant smoothing of the singularity $p\in C(E)$. Furthermore, if $d=8$, then the Milnor fibre of any $\zz$-equivariant smoothing is isomorphic to $M_{i}$.
\end{theo}

\begin{proof}
First, let us assume $E$ is a smooth elliptic curve of degree eight. As described above, a smoothing of $p\in C(E)$ can be obtained by considering one of the two associated del Pezzo surfaces: let us choose $\pl\times\pl$. The description of how to obtain a smoothing of a simple elliptic singularity given at the beginning of this section, for this specific case, gives us the following set up: $\pl\times\pl$ is embedded in $\mathbb P^8$ via the anticanonical divisor, $\pl\times\pl \xrightarrow{i} \mathbb P(H^0(-K_{\pl\times\pl}))$, hence its projective cone is embedded in $\mathbb P^9$. Now, let us consider the involution $j$ defined on this surface by $(z,w)\mapsto (z^{-1},-w)$ which, in homogeneous coordinates, becomes 
$$(x_0:x_1),(y_0:y_1)\mapsto (x_1,x_0),(y_0:-y_1)$$
or, after a change of coordinates
$$(x_0:x_1),(y_0:y_1)\mapsto (x_0:-x_1),(y_0:-y_1)$$
Now we can use this involution to define a $\zz$-action on the line bundle $-K_{\pl\times\pl}=\mathcal O(2,2)$. More precisely we define an involution $\vartheta$ on its global sections $z_0,\dots,z_8$ as the map directly induced by $j_0$ on the bi-homogeneous monomials of degree $(2,2)$ corresponding to each $z_i$ composed with the map $z_i\mapsto -z_i$. Note that this involution fixes four of the nine monomials and acts as $(-1)$ on the remaining five. Moreover we use the composition with the latter map to obtain a $\zz$-action that is non trivial on the fibres above the four points in $\pl\times\pl$ fixed by $j$. The map $\vartheta$ gives an involution on $\mathbb P^8_{z_0,\dots,z_8}$ therefore one on $\mathbb P^9_{z_0,\dots,z_9}$, extending the former involution so that it acts trivially on the last coordinate. In order to show that there exists an equivariant smoothing of $p\in C(E)$ we just need to show that we can construct a family of planes $\mathcal H_t$ in $\mathbb P^9$ which is equivariant with respect to the action just defined on $\mathbb P^9$ and when intersected with $X:=i(\pl\times\pl)\subset \mathbb P^8$ gives the smooth elliptic curve $E$. To make the family $\mathcal H_t$ equivariant, we define it by setting the coefficients $a_i(t)$ of the five coordinates which are not fixed by the involution to be equal to 0 for all $t$. Then $\hat {\mathcal H}_t=\mathcal H_t\cap \mathbb P^8$ is of the form $a_1(t)z_1+a_3(t)z_3+a_5(t)z_5+a_7(t)z_7=0$, and its generic intersection with $i(\pl\times\pl)$ is smooth. Indeed the linear system associated to $\hat{\mathcal H}_t$ in $H^0(-K_{\pl\times\pl},\C)$ is generated by the sections $x_0^2y_0y_1,x_1^2y_0y_1,x_0x_1y_0^2,x_0x_1y_1^2$. The base points are the fixed points on $\pl\times\pl$ and it can be checked directly that $\hat{\mathcal H}_t$ is smooth at these points for generic coefficients $a_i(t)$. Thus we can conclude it is smooth everywhere by Bertini's theorem and it follows that $\mathcal X_t:=\mathcal H_t\cap \overline{C(X)}\subset \mathbb P^9$ is an equivariant smoothing of the simple elliptic singularity of degree 8.

Every del Pezzo surface of degree $8-i$ can be obtained from $\pl\times\pl$ through the blowup of $i$ points satisfying some suitable conditions, thus the same construction used above will provide the required equivariant smoothing. In order to see that the blowups mentioned above are possible we will use a version of the Torelli's theorem for anticanonical pairs with smooth boundary that will be contained in a paper in preparation by Jennifer Li, but we could also use the (weaker) results contained in section 6 of a paper by McMullen (see \cite{M07}). We include the statement of the theorem at the end of this proof for clarity (theorem \ref{torsmooth}) and we give a proof for the case where $i=6$, i.e. the blow up of $\pl\times\pl$ in $6$ points. Indeed, once this is proved, the other cases can be obtained from this.

Now, let $(Y,E)$ be the del Pezzo surface with the (anticanonical) elliptic curve $E$ of degree 2 on it. In order to apply theorem \ref{torsmooth} with $E_1=E_2=E$, $g=-1$, $f^*=\theta$, $\phi\circ\theta=-\phi$ we need to choose an appropriate period point $\phi: E^\perp\rightarrow \textnormal{Pic}^\circ E\cong E$. Let us first establish the notation for a basis of $E^\perp$. Call $F_1,F_2$ the strict transforms of the generators of the Picard group of $\pl\times\pl$ and call $E_1,\dots,E_6$ the exceptional divisors of the blowups. Then a basis for $E^\perp$ is given by $F_1-F_2, F_1-(E_1+E_2), E_1-E_3, E_2-E_4, E_3-E_5, E_4-E_6, E_5-E_6$; for ease of notation let us label them $\alpha_1,\dots,\alpha_7$. Define an involution $\theta$ on the Picard group of $Y$ so that $\theta(F_1)=F_1$, $\theta(F_2)=F_2$, $\theta(E_{2k-1})=E_{2k}$ for $k=1,2,3$. As a consequence $\theta$ acts on $E^\perp$ as the block diagonal matrix

$$\begin{pmatrix}   1 &   &     &     &0 \\
                      & 1 &     &     &  \\
                      &   & J_1 &     &  \\
                      &   &     & J_2 &  \\
                    0 &   &     &     & -1 \\
                                        \end{pmatrix} 
						  \qquad \textnormal{where} \quad J_i=\begin{pmatrix}0 & 1\\
                                             1 & 0\end{pmatrix}
                          \quad\textnormal{for} \ i=1,2$$
With this set up, it remains to show that we can choose the period point $\phi$ such that
\begin{itemize}
\item [a)] $\phi\circ\theta = -\phi$
\item [b)] $\phi(\alpha)\neq 0$ for all $\alpha\in\Phi$, where $\Phi$ is the set of roots in Pic $Y$ (cfr. definition 1.6 in \cite{GHK15b}).
\end{itemize}
Condition (a) translates as:
\begin{equation}\label{eqinv}\begin{cases}
\phi(\alpha_i)=-\phi(\alpha_i) \quad\textnormal{for} \ i=1,2 \\
\phi(\alpha_{2k-1})=-\alpha_{2k} \quad\textnormal{for} \ k=2,3
\end{cases}\end{equation}
As for condition (b), recall that, for any root in $\alpha\in\Phi$, we must have
$$\phi(\alpha)=\sum_i m_i\phi(\alpha_i) \quad\textnormal{with} \ m_i\in\z \quad \forall i$$
Thus (b) becomes:
$$\sum_i m_i\phi(\alpha_i)\neq 0 \quad\textnormal{in} \ E \quad\Longleftrightarrow \quad \sum_i m_i\phi(\alpha_i)+\sum_j n_j\lambda_j\neq 0 \quad\textnormal{in} \ \C$$
given that $E=\C/\Lambda$ and $\Lambda=\z\lambda_1+\z\lambda_2$. From (\ref{eqinv}) it follows that we can choose $\phi(\alpha_3),\phi(\alpha_5), \phi(\alpha_7)$ arbitrarily: let us choose them and $\lambda_1,\lambda_2$ so that they are linearly independent over $\mathbb Q$. Moreover, we can always choose $\phi(\alpha_1),\phi_(\alpha_2)$ so that they are 2-torsion but non zero in $E$ and such that $\phi(\alpha_1)+\phi(\alpha_2)\neq 0$ in $E$. Thus the only linear relations equal to zero we have to check are $\phi(\alpha_3)+\phi(\alpha_4)$ and $\phi(\alpha_5)+\phi(\alpha_6)$. However neither $\alpha_3+\alpha_4$ nor $\alpha_5+\alpha_6$ are roots (their self intersections are equal to $-4\neq-2$), thus we can always choose a $\phi$ as desired ad apply theorem \ref{torsmooth}. We get an isomorphism $f$ that is an involution thanks to its uniqueness.

On the other hand, let us consider a simple elliptic singularity $p\in C(E)$ of degree eight, let $\iota$ be the involution defined on it and let us assume that there exists a $\zz$-equivariant smoothing of this singularity whose Milnor fibre is isomorphic to $M_{ii}$. Then $\iota$ induces an involution on $M_{ii}$ which can be extended to an involution of $\mathbb F_1$. Now theorem 4.2 in \cite{P74} still holds true if we add the hypothesis that the singularity is endowed with a $\zz$-action, because the latter commutes with the $\C^*$-action: more precisely we have that given a $\zz$-equivariant infinitesimal deformation of the cone $C$ as described in \cite{P74}, then it lifts to a $\zz$-equivariant infinitesimal deformation of its projective version. Note that in this case $T^1_C$ would have a $\z\oplus\zz$ grading. In our case we this get that a $\zz$-equivariant smoothing can be globalized to a $\zz$-equivariant smoothing of the projective cone over the elliptic curve, with general fiber $\mathbb F_1$. Now since there is a unique $(-1)$-curve on $\mathbb F_1$, it is necessarily fixed by the involution and can be blown down to obtain an involution on $\mathbb P^2$. The involution on $\mathbb F_1$ has only isolated fixed points (by upper semicontinuity of fiber dimension applied to the fixed locus in the family, since  the involution has isolated fixed points on the special fiber). It follows that the involution on $\mathbb P^2$ has isolated fixed points, which is a contradiction (by the classification of involutions of $\mathbb P^2$).
\end{proof}

\begin{theo}[Li, to appear]\label{torsmooth}
Let $X$ be a smooth projective surface and $E$ a smooth connected anticanonical divisor. Then $E$ is an elliptic curve and $X$ is rational. Let $\phi: E^\perp\rightarrow \textnormal{Pic}^\circ E\cong E$ the period point. Given any such pair $(X_0,E_0)$, if $(X,E)$ is deformation equivalent to it, identify $\textnormal{Pic}X=H^2(X,\z)\cong H^2(X_0,\z)=\textnormal{Pic}X_0$ via parallel transport. We get a map $L\rightarrow E$, with $L=E_0^\perp$. Then:
\begin{itemize}
\item [1.] The period map is surjective, i.e. all homomorphisms $\phi$ arise in this way (for all $E$).
\item [2.] If $(X_1,E_1)$ and $(X_2,E_2)$ are two such pairs, $\theta\in\textnormal{Adm}$ and
\begin{itemize}
\item [i.] There is an isomorphism $g:E_2\rightarrow E_1$
\item [ii.] $\phi_1=g\circ\phi_2\circ\theta$
\item [iii.] There does not exist any $(-2)$-curve in $X_i\setminus E_i$, $i=1,2$. Eequivalently, $\phi_i(\alpha)\neq 0$ for all roots $\alpha\in\Phi_i$.
\end{itemize}
Then there exists an isomorphism $f:(X_1,E_1)\rightarrow(X_2,E_2)$ such that $f^*=\theta$. Moreover $f$ is unique if $\pi_1(X_i)=0$ (iff $X_i$ are not isomorphic to $\mathbb P^2$ or $\pl\times\pl$).
\end{itemize}
\end{theo}

On the other hand if $(Y_{(d)},D_{(d)})$, for $d=2,4,6$, and $(Y_i,D)$, $(Y_{ii},D)$, where $D$ has length $8$, are the Looijenga pairs corresponding to the Milnor fibres of the smoothings of the appropriate simple elliptic singularity, then we can prove what follows.

\begin{prop}\label{ellpairs}
The Looijenga pairs $(Y_{(d)},D_{(d)})$ with $d=2,4,6$, and $(Y_i,D)$ admit an antisymplectic involution which is free away from the anticanonical divisor and acts on it as a reflection, while the pair $(Y_{ii},D)$ does not admit an involution with this properties.
\end{prop}

\begin{proof}
To see that $(Y_i,D)$ admits an involution with the properties described in the statement of the proposition, we follow the steps of proposition \ref{n10} to lift the involution given on $(\pl\times\pl,\Delta)$ to an involution on $(T_{(d)},D_{(d)})$ for $d=4,6$ and $(T_i,G_i)$ for $d=8$. Now, since the interior blowups are performed in a symmetric way in each of these cases, this involution lifts to one on $(Y_{(d)},Y_{(d)})$ and $(Y_i,D)$ respectively, again following proposition \ref{n10}. In order to construct an involution on the Looijenga pair $(Y_{(2)},D_{(2)})$ where $D_{(2)}$ is a cycle of two smooth rational curves of self intersection $-2$ we can proceed as follows. Let $\yd$ be the Looijenga pair obtained from $(\pl\times\pl,\Delta)$ performing four interior blowups $\Delta_2,$ and $\Delta_4$ each and one interior blowup on $\Delta_1$ and $\delta_3$. Then, again following proposition \ref{n10} we can lift the involution on $(\pl\times\pl,\Delta)$ to a new involution on this $\yd$. This map is such that the two $(-1)$ curves contained in $D$ are in the same orbit, therefore we can contract them both and obtain map of Looijenga pairs $\yd\rightarrow (Y_{(2)},D_{(2)})$ and an induced involution on $(Y_{(2)},D_{(2)})$, as desired.

Now consider $(Y_{ii},D)$, suppose it admits an involution $\iota$. First note that, by the condition $\phi(i(q))=(\phi(q))^{-1}$ for all $q\in\langle D_1,\dots,D_8\rangle$ of proposition \ref{modulispaceinv}, we have $\phi(D)=\pm1$. It follows that there is an elliptic fibration $f:Y\rightarrow\pl$ such that $D$ is a fibre of $f$ of multiplicity 1 or 2 (for $\phi(D)$=1 or $-1$ respectively). First suppose that $D$ is a fibre of multiplicity 1. The involution defined on $(Y_{ii},D)$ gives an involution of the elliptic fibration, which on the base space $\pl$ is given by $z\mapsto -z$. Note that the action on $\pl$ cannot be the trivial one, since we are assuming that the involution $\iota$ has no fix points on $Y_{ii}\setminus D$. The $\zz$-action we obtain on the fibration must have two fixed fibers: by assumption one of them has to be the anticanonical divisor $D$, we claim that the other fixed fiber has to be a smooth elliptic curve. Indeed, it cannot be a node or a cusp, since an involution on these singular curves would have to fix the singularity necessarily, contradicting the assumption; on the other hand this fiber cannot be reducible, namely a union of curves of self intersection $-2$, since the only curves with self intersection equal to $-2$ are the ones contained in $D$. Therefore the second fixed fiber $F$ is a smooth elliptic curve where the involution acts as a translation. Now let us consider the quotient of $f:Y\rightarrow \pl$ by the $\zz$-action. After resolving the $A_1$ singularities, we get a new elliptic fibration $\hat f: Z\rightarrow S$, where $S$ has two points corresponding to the fixed points in $\pl$. The associated fibers are a smooth elliptic curve, which is given by the quotient of $F$ and a new reducible fibre $B$ whose dual graph looks as in figure \ref{quotcusp}: define $V:=Z\setminus B$. Consider the exact sequence in relative cohomology:

$$\cdots\rightarrow H^2(Z,V)\rightarrow H^2(Z)\rightarrow H^2(V)\rightarrow H^3(Z,V)\rightarrow \cdots$$ 

We have $H^2(Z,V)\cong H_2(B)\cong \z^9$ and $H^3(Z,V)\cong H_1(B)\cong 0$, therefore $H^2(V)\cong\Pic (Z)/\langle B_1,\dots,B_5,E_1,\dots,E_4\rangle$ and on the other hand $\textnormal{Tors} \ H_1(V)\cong \textnormal{Tors}\ H^2(V)$. Finally, if $G=-K_Z$ then the smooth fiber $F$ has multipicity 2 with $F=2G$ and, as sublattices of $\Pic(Z)$, $F\cong \widetilde D_8$ and $G^\perp\cong \widetilde E_8$, hence we get the inclusions of finite index 2: $\widetilde D_8\subset \widetilde D_8+\z G\subset \widetilde E_8$. Thus from the short exact sequence
$$0 \rightarrow G^\perp/\widetilde D_8\rightarrow \Pic Z/\widetilde D_8 \rightarrow \z$$
we obtain $\textnormal{Tors} \ H_1(V)\cong \textnormal{Tors} \ \Pic(Z)/\widetilde D_8\cong \textnormal{Tors} \ G^\perp/\widetilde D_8$ and the latter has order 4. The space $U:=Y_{ii}\setminus D$ however has trivial fundamental group, therefore the map $U\rightarrow V$ is actually the universal cover map for the space $V$. Since it is a degree two normal cover, then the index of $\pi_1(U)$ in $\pi_1(V)$ has to be equal to two, thus giving $|\pi_1(V)|=2$. This contradicts our previous conclusion that $|\textnormal{Tors} \ (H_1(V))|=4$. Finally consider the case that $D$ is a fibre of $f$ of multiplicity 2. Then, considering the quotient as above, we obtain a rational elliptic surface with two multiple fibres, which is impossible.

\end{proof}

These two propositions combined show that if a simple elliptic singularity admits an equivariant smoothing then there exists a corresponding negative semidefinite Loojienga pair equipped with an antisymplectic involution and vice versa.

\medskip
\bibliography{biblio}
\bibliographystyle{alpha}

\vspace{1cm}
\textit{Angelica Simonetti,} \textsc{Department of Pure Mathematics and Mathematical Statistics, Centre for Mathematical Sciences, University of Cambridge, Wilberforce Road, Cambridge, CB3 0WB}

\textit{Email address:} \texttt{as3220@cam.ac.uk}

\end{document}